\documentclass[a4paper,10pt]{article}
\usepackage{latexsym,amsfonts,amssymb,amsmath}
\usepackage{amsthm}
\usepackage{mathrsfs}
\usepackage{diagrams}

\usepackage[a4paper,hmargin=1in,vmargin={1.55in,1.10in}]{geometry}
\usepackage{diagrams}

\DeclareMathOperator{\Pic}{Pic}

\DeclareMathOperator{\coker}{coker}
\DeclareMathOperator{\Sym}{Sym}
\DeclareMathOperator{\sym}{sym}
\DeclareMathOperator{\Tors}{Tors}
\DeclareMathOperator{\di}{div}
\DeclareMathOperator{\tr}{tr}
\DeclareMathOperator{\herm}{herm}
\DeclareMathOperator{\pim}{Im}
\DeclareMathOperator{\pr}{Re}

\DeclareMathAlphabet{\mathpzc}{OT1}{pzc}{m}{it}

\newcommand{\mbb}{\mathbb}

\newcommand{\mc}{\mathcal}

\newcommand{\ddt}{\frac{d}{dt}}

\newcommand{\dd}[2]{\frac{d #1}{d #2}}
\newcommand{\de}[2]{\frac{\partial #1}{\partial #2}}

\newcommand{\dbar}{\bar{\partial}}

\newcommand{\id}{{\rm id}}
\newcommand{\Hom}{{\rm Hom}}

\newcommand{\Mor}{{\rm Mor}}

\newcommand{\rest}{ {\bigg |} }
\newcommand{\trest}{{\big |} }

\newcommand{\M}{\mathcal{M}}
\newcommand{\matrice}[2]{\left(\begin{array}{#1} #2 \end{array} \right)}

\newcommand{\nota}{\footnote}
\newcommand{\mf}{\mathfrak}

\newcommand{\End}{\mathrm{End}}
\newcommand{\Aut}{\mathrm{Aut}}
\newcommand{\Cl}{\mathrm{Cl}}

\newcommand{\tens}{\otimes}

\newcommand{\w}{\wedge}

\newcommand{\Herm}{\mathrm{Herm}}

\newcommand{\Stab}{\mathrm{Stab}}
\newcommand{\Met}{\mathrm{Met}}

\newtheorem{theorem}{Theorem}[subsection]
\newtheorem{lemma}[theorem]{Lemma}
\newtheorem{pps}[theorem]{Proposition}

\newtheorem{stat}[theorem]{Statement}
\theoremstyle{definition}
\newtheorem{definition}[theorem]{Definition}

\theoremstyle{remark}
\newtheorem{remark}[theorem]{Remark}

\numberwithin{equation}{section}

\mathchardef\phi="0127
\mathchardef\varphi="011E

\mathchardef\alpha="710B
\mathchardef\beta="710C
\mathchardef\gamma="710D
\mathchardef\delta="710E
\mathchardef\epsilon="7122
\mathchardef\zeta="7110
\mathchardef\eta="7111
\mathchardef\theta="7112
\mathchardef\iota="7113
\mathchardef\kappa="7114
\mathchardef\lambda="7115
\mathchardef\mu="7116
\mathchardef\nu="7117
\mathchardef\xi="7118
\mathchardef\pi="7119
\mathchardef\rho="711A
\mathchardef\sigma="711B
\mathchardef\tau="711C
\mathchardef\upsilon="711D
\mathchardef\chi="711F
\mathchardef\psi="7120
\mathchardef\omega="7121
\mathchardef\varepsilon="710F
\mathchardef\vartheta="7123
\mathchardef\varpi="7124
\mathchardef\varrho="7125
\mathchardef\varsigma="7126

\title{Perturbations of the metric in Seiberg-Witten equations}
\author{Luca Scala}

\date{}
\begin{document}
\maketitle
\begin{abstract}
Let $M$ a compact connected orientable 4-manifold. We study the space
$\Xi$ of $Spin^c$-structures of fixed fundamental class, 
as an infinite dimensional 
principal bundle on the manifold of riemannian metrics on $M$. 
In order to study perturbations of the metric in Seiberg-Witten
equations, we study the transversality of  universal equations, parametrized
with all $Spin^c$-structures $\Xi$.
  We prove that, on a complex
K\"ahler surface, for an hermitian metric $h$ sufficiently close to the 
original K\"ahler metric, the moduli space of Seiberg-Witten equations relative to the metric $h$ is smooth of the expected dimension.
\end{abstract}
\section*{Introduction}
Let $(M,g)$ a compact connected  oriented riemannian $4$-manifold. Chosen on $(M,g)$ a $Spin^c$-structure~$\xi$ of spinor bundle $W=W_+ \oplus W_-$ and of determinant line bundle $L\simeq \det W_{\pm}$, consider the Seiberg-Witten equations: 
\begin{subequations}\label{swintro}
\begin{gather}\tag{$SW_{\xi}$a}D^{\xi}_A \psi =  0 \\\tag{$SW_{\xi}$b} \rho_{\xi}(F_A^{+}) =[\psi^* \tens \psi]_0 \;,
\end{gather}
\end{subequations}in the unknowns 
$(A, \psi) \in \mc{A}_{U(1)}(L) \times \Gamma(W_+) $, where $\mc{A}_{U(1)}(L)$ denotes the affine space of $U(1)$-connections on $L$. The aim of this article 
is the study of the behaviour of Seiberg-Witten equations (see \cite{Witten1994}, \cite{SeibergWitten1994a}, \cite{SeibergWitten1994b}, \cite{MorganSWEAFMT}) under perturbations of the metric $g$. 

In Donaldson's theory of $SU(2)$-istantons, deeply related to Seiberg-Witten theory, 
the behaviour of ASD equations $F_A^+=0$ when changing the metric is well understood, and metric
perturbations are the main tool to obtain transversality results: the
celebrated Freed-Uhlenbeck theorem (\cite{DonaldsonGFM},
\cite{FreedUhlenbeckIFM}) states that, for a generic metric, the functional 
defining ASD equations is transversal to the zero section at irreducible connections: consequently, for a generic metric, the moduli space of irreducible 
istantons is smooth of the expected dimension. 

On the other hand,  an analogous  result in Seiberg-Witten theory is unknown; more generally, no much is known on the dependence on the metric of Seiberg-Witten equations, 
one of the reasons being probably the fact that the transversality for equations
($SW_{\xi}$) on an irreducible monopole can be very easily obtained by perturbing the second equation adding a generic selfdual imaginary $2$-form $\eta$. 
The dependence of the metric in Seiberg-Witten equations has been studied by Maier in
\cite{Maier1997}, but always for a generic connection $A$ and no transversality issue
is addressed.
The problem of transversality with perturbation just of the metric
appears in the work of
Eichhorn and Friedrich (see \cite{EichhornFriedrich1997}, reported also in 
\cite{FriedrichDORG} and cited in \cite{Bennequin1997}): the authors 
claim to 
give a positive answer, but their proof is not correct: we will discuss the reason 
in remark \ref{EF}. The purpose of this article is to establish an analogous of 
Freed-Uhlenbeck theorem in Seiberg-Witten theory, giving a correct proof of the 
fact that, for generic metric, the Seiberg-Witten functional is transversal to the zero section and hence that  the 
 moduli space is smooth of the expected dimension (at least on irreducible monopoles).

In order to write Seiberg-Witten equations on the oriented 
riemannian $4$-manifold $(M,g)$, we have to fix a $Spin^c$-structure, that is, 
an equivariant lifting $\xi: Q_{Spin^c(4)} \rTo 
P_{SO(g)}$ of the $SO(4)$-principal bundle of equioriented orthonormal frames for the metric $g$ to a 
$Spin^c(4)$-principal bundle $Q_{Spin^c(4)}$. Since the $Spin^c$-structure is a 
metric concept, that is, it actually determines the metric, 
 when changing the metric on $M$ 
we are forced to change $Spin^c$-structure; however, for different metrics the $SO(4)$-bundles of equioriented orthonormal frames are isomorphic and can be 
lifted to the same principal bundle $Q_{Spin^c(4)}$ (of course by means of different morphisms $\xi$):
consequently, we can fix the bundle $Q_{Spin^c(4)}$ once for
all. Therefore it turns out that the right setting to study perturbations of the metric in Seiberg-Witten equations is considering  \emph{universal equations}, parametrized 
by all $Spin^c$-structures $\Xi$ of fundamental class $c$ and $Spin^c$-bundle $Q_{Spin^c(4)}$; we will then characterize only in a second step the variations of $Spin^c$-structures coming from a variation just
of the metric. The space $\Xi$  can be given the structure of a
(trivial infinite dimensional) principal bundle
over the space of riemannian metrics~$\Met(M)$, of structural group 
$\Aut(Q_{Spin^c(4)} \times_{Spin^c(4)} SO(4) )$; on the principal bundle
$\Xi$ there can now be defined a \emph{natural connection,
the horizontal distribution being characterized by consisting
precisely of variations of $Spin^c$-structures coming from variations
just of the metric}. The connection thus defined turns out to have
nontrivial curvature: it is therefore impossible to find (even
locally) a parallel section $\Met(M) \rTo \Xi$, by means of which
parametrizing correctly Seiberg-Witten equations with just the
metric. 
This is however not a difficult issue, since the universal Seiberg-Witten moduli space $\mc{M}$ admits a $\Aut(Q_{Spin^c(4)})$-equivariant fibration $\mc{M} \rTo \Xi$ over the space of 
$Spin^c$-structures~$\Xi$; hence the transversality of equations 
($SW_{\xi}$) at the point $\xi$ does not depend on the $Spin^c$-structure $\xi$, but only of the metric $g_{\xi}$ compatible with $\xi$: consequently the problem of transversality of Seiberg-Witten equations for generic metrics is equivalent to the problem of transversality of these equations for generic $Spin^c$-structures. 

\sloppy
This formalism (appeared first in \cite{OkonekTeleman1996}) allows us to reduce the problem --- after completing Fr\'echet spaces to 
Sobolev ones in a standard way ---
to the proof of the surjectivity of the differential 
$D_{(A, \psi, \xi)} \mbb{F}_{\Met(M)}$
of the functional $$\mbb{F}_{\Met(M)}(A, \psi, \xi) = (D_A^{\xi} \psi \; , \; 
\rho_{\xi}(F_A^{+,g_{\xi}}) -[\psi^* \tens \psi]_0 ) \;, $$defining universal 
Seiberg-Witten equations, at the solution $(A, \psi, \xi) \in \mc{A}_{U(1)}(L) \times \Gamma(W_+) \times \Xi$. 
To compute this operator, we need to compute 
the variation of the Dirac operator, 
performed first by Bourguignon and Gauduchon \cite{BourguignonGauduchon1992}. 
We present here a simple alternative proof: 
our approach has the advantage of fixing once for all the bundle 
$Q_{Spin^c}$ and consequently the bundle
of spinors $W$, 
and is particularly adapted to the transversality problem: indeed in 
this way all Dirac operators act on the same space of global sections, without any need of 
delicate identifications or transmutations operators.  

\sloppy The surjectivity of the differential $D_{(A, \psi, \xi)}\mbb{F}_{\Met(M)}$ is equivalent to the injectivity of the formal adjoint $(D_{(A, \psi, \xi)}\mbb{F}_{\Met(M)})^*$: consequently nontrivial solutions  of the kernel equations $(D_{(A, \psi, \xi)}\mbb{F}_{\Met(M)})^*u=0$ represent the obstruction to the transversality of the functional $\mbb{F}_{\Met(M)}$. In the general case the equations are 
intricate and we still do not have 
the answer. 

When $(M,g,J)$ is a complex K{\"a}hler surface with complex structure $J$ and with canonical line bundle
$K_M$, the Seiberg-Witten equations admit an interpretation
in terms of holomorphic couples $(\dbar_A, \alpha)$, where
$\dbar_A$ is a holomorphic $(0,1)$-semiconnection on a line bundle
$N$ such that $K^*_M \tens N^{\tens^2} \simeq L$, and $\alpha$ is a
holomorphic section of $(N, \dbar_A)$. This facts allow a drastic
simplification of the Seiberg-Witten equations and consequently of the problem of transversality for generic metrics. 
After interpreting all the preceding 
objects in the context of complex geometry, and thanks to the splitting
of the symmetric endomorphisms with respect to the metric 
into hermitian and anti-hermitian ones, 
the kernel equations $(D_{(A, \psi, \xi)}\mbb{F}_{\Met(M)})^*u=0$ become 
much simpler. Indicating with $\mc{M}_{H_J(M)}$ 
(with $\mc{M}_{K_J(M)}$) the moduli space of hermitian (k\"ahlerian)
monopoles -- that is, monopoles
$[A, \psi, \xi]$ such that $g_{\xi}$ is an hermitian (k\"ahlerian) metric on 
$M$ -- \emph{we proved that the moduli space $\mc{M}_{H_J(M)}$ is smooth at irreducible 
k\"ahlerian monopoles $\mc{M}^*_{K_J(M)}$}. In other words, we get that Seiberg-Witten equations are transversal
for a generic hermitian metric sufficiently close to the K{\"a}hler
metric $g$. We precisely proved: \\
{\bf Theorem.} {\it Let $(M,g,J)$ a K{\"a}hler surface. 
Let $N$ a hermitian line
  bundle on $M$ such that $2 \deg(N) -\deg(K_M) \neq 0$.  Consider the $Spin^c$-structure $\xi$ given by the canonical
  $Spin^c$-structure on $M$ twisted by the hermitian line bundle $N$. 
For a generic metric $h$ in a small open neighbourhood of $g \in \Met(M)$ and for all $Spin^c$-structure $\xi^{\prime}$, 
compatible with $h$,  
 the Seiberg-Witten moduli space $\mc{M}_{\xi^{\prime}}^{SW}$
  is smooth. Actually, the statement holds for a generic 
hermitian metric $h$ in a small open neighbourhood of $g$. }

\vspace{0.3cm}
\noindent
{\bf Acknowledgements.} The results presented here were obtained in 
the first part of my Ph.D. thesis \cite{ScalaPhD} at University Paris 7, under the direction of Prof. Joseph Le Potier. I will never forget his unvaluable help, 
his encouragement, his way of doing Mathematics. 

I would like to thank Prof. Andrei Teleman for his interest in this work, for his precious 
suggestions and for pointing out reference \cite{OkonekTeleman1996}, which 
improves and clarifies the approach of \cite{ScalaPhD}.

\section{$Spin^c$-structures and metrics}\label{section2}
The aim of this section is to recall the basics on $Spin^c$-structures 
from a point of view adapted to the study of metric perturbations, and to 
describe the set of all  $Spin^c$-structures of fixed type and fundamental class
as a principal fibration over the space of metrics on the manifold. See also~\cite[section 2]{OkonekTeleman1996}.

\subsection{$Spin^c$-structures}
Let $n \in \mbb{N}$, $n \geq 1$. 
Recall the fundamental central extensions of groups:
\begin{subequations}\label{exact}
\begin{gather}\label{exact1}
0 \rTo \mbb{Z}_2 \rTo Spin^c(n) \rTo^{\nu:=(\mu, \lambda)} SO(n) \times S^1 \rTo 1 \;. \\
\label{exact2}
1 \rTo S^1 \rTo Spin^c(n) \rTo^{\mu} SO(n) \rTo 1 
\end{gather}
\end{subequations}
Let now $M$ be a compact connected oriented manifold of dimension $n$ and 
let $P_{GL_+(n)}$ the principal $GL_+(n)$-bundle of oriented frames 
of the tangent bundle $TM$.
\begin{definition}A $Spin^c$-structure on $M$ (of type $Q_{Spin^c(n)}$) is the data of a $Spin^c(n)$-principal 
bundle $Q_{Spin^c(n)}$ over $M$ and of a $\mu$-equivariant morphism: 
$ \xi: Q_{Spin^c(n)} \rTo  P_{GL_+(n)} \;$.
The line bundle $L:= Q_{Spin^c(n)} \times_{\lambda} \mbb{C}$ 
is called the determinant line bundle and its first Chern
class $c:= c_1(L)$ is called the fundamental class of the $Spin^c$-structure $\xi$. 
Two $Spin^c$-structures $\xi \colon Q_{Spin^c(n)}\rTo  
 P_{GL_+(n)}$ and 
 $\xi^{\prime}\colon Q^{\prime}_{Spin^c(n)} \rTo P_{GL_+(n)}$
are isomorphic if there 
exist a $Spin^c(n)$-equivariant morphism $f\colon Q_{Spin^c(n)} \rTo Q^{\prime}_{Spin^c(n)}$ such that $\xi^{\prime} \circ f=\xi$.
\end{definition}It is well known that a $Spin^c$-structure of fundamental class $c$ 
exists if and only if $c 
\equiv w_2(M) \mod 2$. 
\begin{remark}\label{sonequivalent}Given a $Spin^c(n)$-principal bundle $Q_{Spin^c(n)}$,
we can form the $SO(n)$-principal bundle $Q_{SO(n)}:= Q_{Spin^c(n)} \times_{\mu} SO(n)$ and the $U(1)$-principal bundle $Q_{U(1)}:= Q_{Spin^c(n)} \times_{\lambda} U(1)$. Every $\mu$-equivariant morphism $\xi : Q_{Spin^c(n)} \rTo P_{GL_+(n)}$ 
factors 
through the composition of the $2$-fold covering 
projection $\eta :  Q_{Spin^c(n)} \rTo  Q_{SO(n)}$, 
followed by the $SO(n)$-equivariant
embedding 
$\gamma_{\xi}:~Q_{SO(n)} \rInto P_{GL_+(n)}$. 
It is thus clear that, once fixed a $Spin^c(n)$-bundle $Q_{Spin^c(n)}$, the 
data of a $Spin^c$-structure of principal bundle $Q_{Spin^c(n)}$ is equivalent to the data of a $SO(n)$-equivariant embedding $ Q_{SO(n)} \rInto P_{GL_+(n)}$.
\end{remark}

\begin{remark}It is a fundamental fact that a $Spin^c$-structure $\xi$ is a \emph{metric} concept. Indeed the embedding $\gamma_{\xi}$, induced by a $Spin^c$-structure $\xi$, provides a 
$SO(n)$-reduction of the principal bundle $P_{GL_+(n)}$, 
corresponding to the choice of a riemannian
metric $g_{\xi}$ on $TM$. We will denote with 
$P_{SO(g_{\xi})}$ the image of $\gamma_{\xi}$, that is, the principal $SO(n)$-subbundle of 
$P_{GL_+(n)}$ consisting of $g_{\xi}$-orthonormal oriented frames and with 
$\alpha_{\xi}$ the lifting $\alpha_{\xi}: Q_{Spin^c(n)} \rTo P_{SO(g_{\xi})}$. We will say that the metric $g_{\xi}$ is compatible with the $Spin^c$-structure 
$\xi$.

\end{remark}
\subsection{$Spin^c$-structures of fixed type and fundamental class}
Let now $c \in H^2(M,\mbb{Z})$ such that $c \equiv w_2(M) \mod 2$. The long 
non-abelian cohomology sequences associated to the central extensions
 (\ref{exact}) read, for the $H^1$-level, 
identifying $H^1(M, S^1)$ with $ H^2(M, \mbb{Z})$:
\begin{diagram}[LaTeXeqno,height=0.65cm]\label{lecs}
 {\qquad}& \rTo & H^1(M, \mbb{Z}_2) & \rTo^b & H^1(M, Spin^c(n)) &\rTo &
 H^1(M,SO(n)) \times H^2(M,\mbb{Z}) &\rTo & H^2(M, \mbb{Z}_2) & \rTo
 & {\qquad}\\
 {\qquad} & & \dTo^{e} & &\dTo^{=} & &\dTo & &\dTo &  &{\qquad}\\
{\qquad} & \rTo & H^2(M, \mbb{Z}) & \rTo^a & H^1(M, Spin^c(n)) &\rTo &
H^1(M,SO(n))  &\rTo & H^2(M, S^1) & \rTo & {\qquad} 
\end{diagram}
Setting $\mf{t}:= \Tors_2 H^2(M, \mbb{Z})= {\rm Im} \: e $ and $\mf{c} := \ker a$, we have the exact sequence:
$$ 0 \rTo \mf{c} \rTo \mf{t} \rTo^a H^1(M, Spin^c(n)) \rTo  H^1(M,SO(n)) \times H^2(M,\mbb{Z}) \rTo H^2(M, \mbb{Z}_2)\;.$$
\begin{remark}\label{rmk:gllift}Since $SO(n)$ is a maximal compact subgroup of $GL_+(n)$,  
$H^1(M , SO(n)) \simeq H^1(M, GL_+(n))$ (cf. \cite[appendix
B]{LawsonMichelsohnSG}); 
hence 
we can replace the first long exact sequence in (\ref{lecs}) with
$$ H^1(M , \mbb{Z}_2) \rTo  H^1(M, Spin^c(n)) \rTo H^1(M,GL_+(n)) \times H^2(M,\mbb{Z}) \rTo 
H^2(M, \mbb{Z}_2) $$
\end{remark}
Fix now a $Spin^c(n)$-bundle $Q_{Spin^c(n)}$  such that its isomorphism class $[Q_{Spin^c(n)}] \in H^1(M, Spin^c(n))$ lifts the couple $([P_{GL_+(n)}],c) \in H^1(M,GL_+(n)) \times H^2(M,\mbb{Z})$. 
For any metric $g$ on $M$ the element 
$[Q_{Spin^c(n)}]$ lifts the couple $([P_{SO(g)}], c) \in H^1(M, SO(n)) \times H^2(M,\mbb{Z}) $, where 
$P_{SO(g)}$ is the principal bundle of  oriented frames in $TM$, orthonormal
for the metric $g$.

We denote with $\Xi$ the space of all $\mu$-equivariant 
morphisms: $\xi: Q_{Spin^c(n)} \rTo P_{GL_+(n)}$ or, equivalently, by remark 
\ref{sonequivalent}, all 
$SO(n)$-equivariant maps: $\gamma: Q_{SO(n)} \rInto P_{GL_+(n)}$: 
$$ \Xi:= \Mor_{\mu}( Q_{Spin^c(n)},P_{GL_+(n)}) \simeq  \Mor_{SO(n)}( 
Q_{SO(n)},P_{GL_+(n)})\; .$$
The space $\Xi$ parametrizes all the $Spin^c$ structures on $M$ of fixed 
type $Q_{Spin^c(n)}$ and fundamental class~$c$.
Since every $Spin^c$-structure determines a metric, 
the space $\Xi$ is fibered over the space of riemannian metrics: 
\begin{equation} \label{fibration}\Xi \rTo \Met(M) \;.
\end{equation}Two $Spin^c$-structures compatibles with the same 
metric differ for the action of $\Aut(Q_{SO(n)})$, hence $\Xi$ has the structure of $\Aut(Q_{SO(n)})$-principal bundle; however, two $Spin^c$-structures 
compatible with the same metric need not to be isomorphic: indeed, they are isomorphic if and only if they differ for the action of $\Aut(Q_{Spin^c(n)})$. The long exact sequence: 
$$ 0 \rTo C^{\infty}(M, S^1) \rTo \Aut(Q_{Spin^c(n)}) \rTo ^{\hat{\eta}}
\Aut(Q_{SO(n)}) \rTo H^1(M, S^1) 
$$induced by the central extension (\ref{exact2}),
implies that the group $\Gamma:= 
{\rm Im} \; \hat{\eta}$
acts in a free and transitive way on the 
fibers of (\ref{fibration})
; moreover, the group 
$ \coker \hat{\eta} \simeq \mf{c}$
parametrizes the set of isomorphism classes of $Spin^c$-structures of fixed type $Q_{Spin^c(n)}$ and
fundamental class $c$ over a fixed metric. The quotient $\Xi/\Gamma$ is 
isomorphic to:
$$ \Xi / \Gamma \simeq \Xi/ \Aut(Q_{Spin^c(n)}) \simeq \Met(M) \times
\pi_0(\Xi/ \Aut(Q_{Spin^c(n)}) \simeq \Met(M) \times \mathfrak{c} \;
,$$because $\Met(M)$ is contractible. If $M$ is simply connected, all
$Spin^c$-structures compatible with the same metric are isomorphic.
\begin{remark}
If $n= \dim M = 4$, the map $b: H^1(M,\mbb{Z}_2) \rTo H^1(M,
Spin^c(4))$ is trivial; consequently $\mf{c} \simeq \mf{t} = \Tors_2
H^2(M, \mbb{Z})$. Hence, for any fixed fundamental class $c \in
H^2(M,\mbb{Z})$, there is a unique isomorphism class $[Q_{Spin^c(4)}]$
lifting the couple $([P_{GL_+(4)}], c) \in H^1(M, GL_+(4)) \times
H^2(M, \mbb{Z})$ in remark \ref{rmk:gllift}. In this case the space $\Xi$ 
parametrizes all the $Spin^c$ structures on $M$ of fixed 
fundamental class~$c$. The connected components of the quotient
$\Xi/\Gamma$ are parametrized by $\mf{t}$.
\end{remark}

\begin{remark}\label{rmk:rim}
The spaces $\Met(M)$ and $\Xi$ can be viewed as spaces of global sections of fiber bundles. Indeed, considered the fiber bundles over $M$: 
\begin{gather} Met(M) := \coprod_{x \in M} \Met(T_xM) \\ 
Mor_{\mu}(Q_{Spin^c(n)} , 
P_{GL_+(n)}) := \coprod_{x \in M} \Mor_{\mu}(Q_{Spin^c(n),x} , 
P_{GL_+(n),x})\end{gather}then $\Met(M) = \Gamma(M, Met(M))$ and 
$\Xi = \Gamma (M, Mor_{\mu} (Q_{Spin^c(n)} , 
P_{GL^+(n)}))$. 
Moreover, for each $x \in M$, the projection 
\begin{equation} \label{pw}
\Mor_{\mu}(Q_{Spin^c(n),x} , P_{GL_+(n),x}) \rTo \Met(T_xM)
\end{equation}
 is a trivial 
finite dimensional principal bundle of 
structural group $SO(n)$, noncanonically 
isomorphic to the $SO(n)$-principal bundle $P_{GL_+(n),x}\rTo \Met(T_xM)$. 
The fiberwise projection (\ref{pw}) induces the 
global projection of fiber bundles
$Mor_{\mu}(Q_{Spin^c(n)} , 
P_{GL_+(n)}) \rTo Met(M)$ and the infinite dimensional principal bundle $\Xi \rTo \Met(M)$. 
\end{remark}

\begin{remark}
The spaces $\Met(M)$, of riemannian metrics over $M$, and $\Xi$, 
of $Spin^c$-structures on $M$ of type $Q_{Spin^c(n)}$, 
are infinite dimensional Fr\'echet manifolds, because spaces of global sections of fiber bundles over $M$, as explained in 
\cite{FreedGroisser1989} and, more recently, in \cite{KrieglMichor1997}. 
The projection $\Xi \rTo \Met(M)$ gives $\Xi$ the structure of an 
infinite dimensional Fr\'echet principal bundle with regular 
Fr\'echet-Lie group  
$\Aut(Q_{SO(n)})$ as structure group (see \cite{KrieglMichor1997}, 
Chapter VIII, \S~38-39).
\end{remark}
\begin{remark}
The manifold $\Met(M)$ of riemannian metrics on $M$, can be equipped
with a natural riemannian metric making it a $\infty$-dimensional Frech\'et
riemannian manifold (see \cite{FreedUhlenbeckIFM}, \cite{GilMedranoMichor1991}).
\end{remark}

\subsection{Changes of metric: the natural connection on $\Xi$.}\label{subsect:natconnection}
The group 
$\Aut(P_{GL_+(n)})$ acts freely and transitively (on the right) on the space 
$\Xi$, hence the choice of an element $\xi \in \Xi$ defines an 
isomorphism: $
\Aut(P_{GL_+(n)}) \simeq \Xi$, defined by $
\phi  \rMapsto  \phi^{-1}\circ \xi $ and
such that $g_{\phi^{-1}\circ \xi}=\phi^*g_{\xi}$. 
The choice of $\xi \in \Xi$ determines 
the polar decomposition: 
$$ \Aut(P_{GL_+(n)}) \simeq \Aut(P_{SO(g_{\xi})}) \times 
\Sym^+(P_{SO(g_{\xi})}) \;, $$where we denoted with 
$ \Sym^+(P_{SO(g_{\xi})}) $ 
the space of functions $f: P_{SO(g_{\xi})} \rTo \Sym^+(n)$ such that $f(pg)=g^{-1} f(p) g= {\;}^tg f(p)g$ for all $g \in SO(n)$ and where $\Sym^+(n)$ is the 
space of positive symmetric automorphisms of $\mbb{R}^n$ with respect to the 
standard scalar product.

Consider the section \begin{diagram}[LaTeXeqno,height=.6cm]\label{numero}
\Met(M) & \simeq & \Sym^+(P_{SO(g_{\xi})}) & \rTo^{\; \; \sigma_{\xi}} & \Xi \\
\phi^* g_{\xi} & \lMapsto & \phi & \rMapsto & \phi^{-1} \circ \xi \end{diagram}
Taking the tangent space of the 
image of this section in $\xi$, 
$ H_{\xi} := T_{\xi}( {\rm im} \,\sigma_{\xi})$, defines in a natural way a 
$\Aut(Q_{SO(n)})$-equivariant horizontal distribution in $T\Xi$ and hence a 
connection on $\Xi$, which we will call the \emph{natural connection} 
on $\Xi$. 
It is then natural, when changing the metric $g \in \Met(M)$ along a path 
$g_t$, \emph{to change the $Spin^c$-structure lifting the path to $\Xi$ in a 
parallel way for the natural connection on $\Xi$ just defined.}
\begin{remark} Since $\Xi \rTo \Met(M)$ is an infinite dimensional 
 principal bundle with regular Lie 
group as structural group, the parallel transport exists and it is unique for any connection on $\Xi$. Moreover the 
curvature of a connection can be 
interpreted, as usual, as the obstruction of the 
integrability of the horizontal distribution. See \cite[Chapter VIII, \S 39]{KrieglMichor1997}, 
 for details.
\end{remark}

\paragraph{Parallel transport on $\Xi$ for the natural connection.}
Let $\xi_0$ a given $Spin^c$-structure in $\Xi$ and $g_t = \phi^*_t
g_{\xi_0}$ a 
path of metrics in $\Met(M)$, 
such that $\phi_t \in \Sym^+(TM,g_{\xi_0})$ and $\phi_0 = \id$.
To determine the equation of the parallel transport of $\xi_0$ along the path $g_t$ for the natural connection, 
consider a path of $Spin^c$-structures $\xi_t$ in $\Xi$ starting from $\xi_0$, and subject to the condition $g_{\xi_t}=g_t$. Writing that the path $\xi_t$ is parallel means that 
$\dot{\xi}_t=d/d\lambda |_{\lambda =t }\; ( \theta_{\lambda}^{-1} \circ \xi_t)$, 
$\theta_{\lambda} \in \Sym^+(TM, g_t)$,
$\theta_{\lambda}^* g_t =g_{\lambda}$; indicating with
$\dot{\theta}(t)= d \theta_{\lambda} /d \lambda \trest_{\lambda=t}$,  we have $\dot{\xi}_t = - \dot{\theta}(t) \circ \xi_t$.
Since $2 g_t \dot{\theta}(t) = 
\dot{\phi}_{t}^* g_{\xi_0}$ and consequently  
$\dot{\theta}(t)= 1/2 ~(\phi^*_tg_{\xi_0} )^{-1}  \dot{\phi}_t ^*g_{\xi_0} $, the parallel transport equation reads: 
$$ \dot{\xi}_t = -\frac{1}{2} \left( (\phi^*_tg_{\xi_0} )^{-1}  \dot{\phi}_t ^*g_{\xi_0} 
\right) \circ \xi_t \;.$$
\paragraph{Curvature.}
The following proposition gives the curvature of the 
natural connection on the principal bundle $\Xi \rTo M$. Similar computations 
have been made poitwisely by Bourguignon and Gauduchon \cite{BourguignonGauduchon1992} to compute the 
curvature of the $O(n)$-principal bundle $L(V) \rTo \Met(V)$ of linear frames of a 
real $n$-dimensional vector space $V$ over its cone of metrics  $\Met(V)$. The natural connection on $\Xi$ induces pointwisely connections on the $SO(n)$-principal bundles $Mor_{\mu}(Q_{Spin^c(n), x}, P_{GL_+(n), x}) \simeq P_{GL_+(n),x} \rTo \Met(T_xM)$. With exactly the same proof as in \cite[Lemma 3]{BourguignonGauduchon1992}, we find that the curvature  of this connection is given by: 
$\Omega_{g_x}(h_x,k_x) = -1/4 ~ [g_x^{-1}h_x, g_{x}^{-1}k_x]$ for $g_x \in \Met(T_xM)$, $h_x, k_x \in S^2T^*_xM$. General facts on vector 
fields of spaces of sections of fiber bundles (see \cite[Appendix]{FreedGroisser1989}) imply 
that the curvature $\Omega$ of the bundle 
$\Xi \rTo \Met(M)$ is pointwisely the curvature of the bundle 
$P_{GL_+(n),x} \rTo \Met(T_xM)$: 
$ (\Omega_g(h,k))_x = \Omega_{g_x}(h_x, k_x)$. Therefore we get:
\begin{pps}\label{curv} The curvature of the natural connection of the principal bundle 
$\Xi \rTo M$ is given by:
$$ \Omega_{g_{\xi}}(h,k) = -\frac{1}{4}[g_{\xi}^{-1}h, g_{\xi}^{-1}k] \;,$$in ${\rm ad}(\Xi)_{g_{\xi}} \simeq \mf{so}(TM, g_{\xi})$, for 
$h,k \in  S^2T^*M \simeq T_{g_{\xi}} \Met(M)$.
\end{pps} 
\begin{remark}\label{curvrmk} Proposition \ref{curv} implies that the horizontal distribution 
$H_{\xi}$ on $T_{\xi}\Xi$ is \emph{never} integrable. Consequently 
\emph{there are no parallel sections $\sigma: \Met(M) \rTo \Xi$ (even locally).}
\end{remark}
\subsection{Parametrized Dirac operators}
Let $m= [n/2]$. Consider an irreducible $\Cl(\mbb{R}^n)$ representation
$ \rho_0 :  \Cl(\mbb{R}^n) \rTo \End(W_0)$ where $W_0$ is an hermitian 
vector space of complex dimension $2^m$. The bundle of spinors is defined by: 
$W := Q_{Spin^c(n)} \times _{\rho_0} W_0$. The choice of an element $\xi \in \Xi$
induces the Clifford multiplication on the bundle of Clifford algebras 
$\Cl(TM)$:
\begin{equation}\label{moltclifford} \rho_{\xi} : \Cl(TM) \simeq  P_{SO(g_{\xi})} \times_{SO(n)} \Cl(\mbb{R}^n) 
  \rTo Q_{Spin^c(n)} 
\times_{\rho_0} \End(W_0) \simeq \End(W) \;;\end{equation}identifying
of $TM$ and $T^*M$ via $g_{\xi}$ we get as well a Clifford moltiplication on $\Cl(T^*M)$, which we will keep on denoting $\rho_{\xi}$.
We have the following diagram:
\begin{diagram}
Q_{U(1)} &         \lTo^{\beta}     &Q_{Spin^c(n)}&                                     & & & \\
 & & \dTo ^{\eta} & \rdTo^{\alpha_{\xi}} \rdTo(4,2)^{\xi} &  & & & \\
 &      &Q_{SO(n)}    & \rTo_{\gamma_{\xi}}   & P_{SO(g_{\xi})}&\rInto & P_{GL_+(n)} 
\end{diagram}\sloppy
Let $\omega_{g_{\xi}} \in A^1(P_{SO(g_{\xi})}, \mathfrak{so}(n))$ be 
the Levi-Civita connection on $P_{SO(g_{\xi})}$: it induces a 
$SO(n)$-equivariant connection form on $P_{GL_+(n)}$ which we keep on denoting $\omega_{g_{\xi}}$.  
Let   $A \in 
A^1(Q_{U(1)}, \mathfrak{u}(1))$ a $U(1)$-connection form on $Q_{U(1)}$. The $Spin^c$-\emph{connection} 
$\Omega_{A, \xi}$ on $Q_{Spin^c(n)}$ is defined as: 
\begin{equation}\label{spinconnection}
\Omega_{A, \xi} := d\nu ^{-1}(\alpha_{\xi}^* \omega_{g_{\xi}} + \beta^*A) \; 
\end{equation}
seen in  $A^1(Q_{Spin^c(n)}, \mathfrak{spin}^c(n)) $, where $d \nu $
is the isomorphism: 
$\mf{spin}^c \simeq \mf{so}(n) \oplus \mf{u}(1)$. The $Spin^c$-connection form 
$\Omega_{A, \xi}$ defines a connection $\nabla^{W,\xi}_A$ on the 
associated vector 
bundle of spinors $W$ in the following standard way. If $p$ is the projection 
$p : Q_{Spin^c(n)} \rTo M$, the vector bundle $p^*W$ trivializes as: 
$p^*W \simeq Q_{Spin^c(n)} \times W_0$.  The connection $\nabla^{W,\xi}_A$ is then characterized by: 
\begin{equation}\label{nablaspin} p^* \nabla^{W,\xi}_A = d + \Omega_{A,\xi} \;, \end{equation}where
 $\Omega_{A, \xi}$ is seen 
 in $A^1(Q_{Spin^c(n)}, \End(W_0))$ and $d$ is the trivial connection. The Dirac operator 
$D_A^{\xi}$ is then  the composition: 
$$ D_A^{\xi} : \Gamma(W) \rTo^{\nabla^{W,\xi}_A} \Gamma(T^*M \tens W) 
\rTo^{\rho_{\xi}} \Gamma(W) \;.$$Hence we have a family of first order differential operators 
\begin{equation}\label{familydirac}
\mathpzc{D} : \Xi \rTo \mathrm{Diff}^1(W) \;,
\end{equation}given by $\mathpzc{D}(\xi)=D^{\xi}_A$ and parametrized by $Spin^c$-structures $\xi \in \Xi$, \emph{all acting on the same 
vector bundle of spinors} $W$.
\subsection{Parametrized Seiberg-Witten equations}\label{2.5}
Let now be $n=4$. The irreducible $\Cl(\mbb{R}^4)$-module $W_0$ splits 
as the direct sum $W_0^+ \oplus W_0^-$ of irreducible 
$Spin^c(4)$-representations: consequently, the bundle of spinors $W$ 
splits as well in the direct sum of positive and negative spinors: 
$W = W_+\oplus W_-$. 
The determinants of $W_+$ and $W_-$ are 
canonically
identified with the determinant line bundle $L = Q_{U(1)}
\times_{U(1)} \mbb{C}$; we indicate with $i$ the identification 
$\det W_+  \simeq \det W_-$.
The Clifford multiplication $\rho_{\xi}$ induces isomorphisms 
$ \Lambda_{\pm} T^*M \simeq \mathfrak{su}(W_{\pm})$, denoted again 
by $\rho_{\xi}$. 
Let $\mc{C} := \mc{A}_{U(1)}(L) \times \Gamma(W_+) $ be the space of unparametrized configurations and 
$\mc{D} := \Gamma(W_-) \times i\mathfrak{su}(W_+)$; we denote moreover with $\mc{C}^* := \mc{A}_{U(1)}(L) \times (\Gamma(W_+)\setminus \{ 0 \}) $ the irreducible configurations.
In order to consider general parameter spaces, let $T$ be a Fr\'echet
splitting\nota{in the sense of \cite[Definition
  27.11]{KrieglMichor1997}} 
submanifold\nota{We will always consider submanifold $T$ 
given by a space of $C^{\infty}$-sections of a fiber subbundle of 
the fiber bundle $Met(M) \rTo M$ consider in remark \ref{rmk:rim}.}
 of the manifold $\Met(M)$ 
of riemannian metrics on~$M$. Let~$\Xi_T$ be the restriction 
of the principal bundle $\Xi$ to the submanifold $T$. 
The natural connection on~$\Xi$ induces a natural 
connection on $\Xi_T \rTo T$.
The \emph{parametrized Seiberg-Witten equations} for unknowns $(A, \psi, \xi) \in 
\mc{C}\times \Xi_T$ are: 
\begin{subequations}\label{PSW}
\begin{gather}\label{psw1}
D_A^{\xi} \psi =  0 \\\label{psw2} 
\rho_{\xi}(F_A^{+, g_{\xi}}) - [\psi^* \tens \psi]_0 =  0
\end{gather}
\end{subequations}We denote with $\mbb{F}_T: \mc{C} \times \Xi_T \rTo \mc{D}$ the functional defining the equations 
and with $\mc{S}_T$ the space of solutions. 
The group $\Aut(Q_{Spin^c(4)})$ acts on
$W$ via isometries, respecting the decomposition in positive and negative spinors and the identification $i$ between the determinants $\det W_+$ and $\det W_-$; if we denote with 
$U(W_+)\times_0 U(W_-)$ the bundle of groups of pairs $(f_+,f_-) \in U(W_+)\times U(W_-)$ with the same determinant, we have $\Aut(Q_{Spin^c(4)}) \simeq 
C^{\infty}(M,U(W_+)\times_0 U(W_-))$. Consequently $\Aut(Q_{Spin^c(n)})$ acts 
(on the right) on $\mc{C} \times \Xi_T$ and $\mc{D}$, setting: 
\begin{gather*}
(A, \psi, \xi)\cdot f := (\beta(f)^* A, f^{-1} \psi, f^*\xi) \\
(\chi, \eta)\cdot f := (f^{-1} \chi, f^{-1} \eta f)
\end{gather*}where $\beta(f)$ denotes the image of $f$ for the morphism 
$\Aut(Q_{Spin^c(4)}) \rTo \Aut(Q_{U(1)})$ induced by the projection $\beta$.
The restriction of this action to the group $\mc{G}:= C^{\infty}(M, S^1) \subset \Aut(Q_{Spin^c(4)})$ coincides 
with the 
classical action of the Seiberg-Witten gauge group 
on $\mc{C}$. The functional
$\mbb{F}_T$ is equivariant for the $\Aut(Q_{Spin^c(4)})$-action;
\begin{equation} \label{equi}\mbb{F}_T((A, \psi, \xi) \cdot f)= \mbb{F}_T(A, \psi, \xi) \cdot f \;.\end{equation}
As a consequence the group $\Aut(Q_{Spin^c(n)})$ preserves the solutions 
$\mc{S}_T$ of the equations (\ref{PSW}). Moreover the actions of $\mc{G}$
 and of $\Aut(Q_{Spin^c(4)})$ commute, so that we can form a 
\emph{parametrized Seiberg-Witten moduli space} $\mc{M}_T:= \mc{S}_T/\mc{G}$, 
fibered over $\Xi_T$: 
\begin{equation}\label{fibbr}
\pi_T: \mc{M}_T \rTo \Xi_T \;,
\end{equation}
equipped with a $\Gamma$-action, making the preceding fibration equivariant.
The fiber $\mc{M}_{\xi}= \pi_T^{-1}(\xi)$ of the fibration (\ref{fibbr}) over $\xi$ is exactly the standard
Seiberg-Witten moduli space for to the $Spin^c$-structure $\xi \in \Xi_T$. 
\section{Variation of the Dirac operator}
This section is devoted to the computation of the variation of the Dirac operator with respect to the metric, by means of the formalism introduced in 
section \ref{section2}:
what we will say in this section 
holds for any $n \geq 1$.  Given a particular 
$Spin^c$-structure $\xi_0 \in \Xi$,  we compute the 
differential $D_{\xi_0}\mathpzc{D}\trest_{H_{\xi_0}}$
in the point $\xi_0$ of the family $\mathpzc{D}$ of differential 
operators (\ref{familydirac}) restricted to the horizontal direction 
$H_{\xi_{0}}$ for the natural connection on $\Xi$; 
this amounts to compute the differential at the 
identity of the family of differential operators: 
\begin{diagram}[height=0.6cm]
 \mathpzc{D} \circ \sigma_{\xi}: 
\Sym^+(TM,g_{\xi}) & \rTo &\mathrm{Diff}^1(W) \\
\qquad \quad \phi & \rMapsto  & D_A^{ \phi^{-1} \circ \xi}
\end{diagram}
Let us compute $D_{\id}(\mathpzc{D} \circ \sigma_{\xi})(s)$ for $s \in \sym(TM, g_{\xi}) \simeq T_{\id}\Sym^+(TM,g_{\xi})$. Consider the path in $\Sym^+(TM,g_{\xi})$ given by 
$\phi_t = \id + ts$ for small $t$. Let $g_t$ be the metric $g_t:= \phi_t^* g_{\xi}$. Let 
$k = dg_t /dt |_{t=0} = 2gs$. Set $\varphi_t = \phi_t^{-1}$, seen in 
$\Aut(P_{GL_+(n)})$. Since by definition of the Clifford multiplication 
on the cotangent bundle $T^*M$, we have $\rho_{\varphi_t \circ \xi} = \rho_{\xi} \circ \varphi^*_t$, 
the searched differential is: 
\begin{eqnarray*} D_{\id}(\mathpzc{D} \circ \sigma_{\xi})(s) & = & 
\ddt \rest_{t=0} D_A^{\varphi_t \circ \xi}  = 
\left( \ddt \rest_{t=0} {\rho_{\xi} \circ \varphi^*_t} \right)  \circ \nabla_A^{W, \xi} + \rho_{\xi} \circ \ddt \rest_{t=0}{\nabla_A^{  W, \varphi_t \circ \xi} }  \\ 
 & = & - \rho_{\xi} \circ s^* \circ \nabla_A^{W, \xi} + 
\rho_{\xi} \circ \dot{\nabla}_A^{W} (s) \;,
\end{eqnarray*}where $\varphi^*_t$ and $s^*$ are the transposed of $\varphi_t$ and $s$, respectively.  We compute now the variation of the 
spinorial connection~$\dot{\nabla}_A^{W} (s)$.
\subsection{Variation of the spinorial connection}
The following lemma contains the wanted result about the variation of the 
spinorial connection. Let $\nabla^{g_{\xi}}$ be the Levi-Civita connection on 
$TM$ for the metric $g_{\xi}$.
\begin{lemma}The differential of the map: $\Xi \rTo  
\mc{A}(W) $ sending $\zeta$ to $\nabla_{A}^{W, \zeta}$
in the point $\xi$ along the horizontal direction corresponding to the variation of the metric $g_{\xi}$ in the direction 
 $s \in \sym(TM , g_{\xi})$, is given by: 
$$  \dot{\nabla}_A^{W} (s) = \frac{1}{2} \rho_{\xi}(
\dot{\nabla}(s) - \nabla^{g_{\xi}} s) \in A^1(M , 
\mathfrak{su}(W) ) $$where $\dot{\nabla}(s)$ denotes the variation of the Levi-Civita connection in the point $g_{\xi}$ along the direction $s$ and  where we see the form 
$\dot{\nabla}(s) - \nabla^{g_{\xi}} s$ in $A^1(M, 
\mathfrak{so}(TM, g_{\xi}))$.
\end{lemma}\begin{proof} \sloppy We first compute the 
variation of the spinorial 
connection form $\Omega_{A, \xi}$ on $Q_{Spin^c(n)}$. Let $\omega_{g_t}$ the Levi-Civita connection form on $P_{SO(g_{t})}$, seen as a connection form in 
$A^1(P_{GL_+(n)}, \mathfrak{gl}(n))$. 
The spinorial connection form is defined as: $$\Omega_{A, \varphi_t \circ \xi} = 
\rho_0 d\nu ^{-1} ((\varphi_t \circ \xi)^* \omega_{g_t} + \beta^*A ) \;.$$
Differentiating this relation in $t=0$ we obtain: 
\begin{eqnarray*} \dot{\Omega}_{A, \xi} & = & \rho_0 d\nu ^{-1} 
\ddt \rest_{t=0} (\varphi_t \circ \xi)^* \omega_{g_t} = \rho_0 d\nu ^{-1} \xi^* \ddt \rest_{t=0} ( \varphi^*_t \omega_{g_t}) \;.
\end{eqnarray*}It is  straightforward to remark that the form $\varphi^*_t \omega_{g_t}$ is defined on $P_{SO(g_{\xi})}$ and is there pseudotensorial of 
type~$(\mathrm{ad}, \mathfrak{so}(n))$; hence its derivative $d (\varphi^*_t \omega_{g_t} )/dt |_{t=0}$ is tensorial of type~$((\mathrm{ad}, \mathfrak{so}(n))$.
Consequently, if $q$ is the projection $q: P_{SO(g_{\xi})} \rTo M$, 
there exists a form $\dot{\omega_M} \in A^1(M, 
\mathfrak{so}(TM, g_{\xi}))$ such that $$ \ddt 
(\varphi^*_t \omega_{g_t} ) \rest_{t=0} = q^*(\dot{\omega}_M) \in A^1(P_{SO(g_{\xi})}, \mf{so}(n))\;.$$ 
Therefore:  \begin{equation*}\dot{\Omega}_{A, \xi} = 
\rho_0 d \nu^{-1} \xi^* q^*  (\dot{\omega}_M) = \rho_0 d
\nu^{-1} p^*( \dot{\omega}_M ) = \frac{1}{2} p^* ( 
\rho_{\xi}( \dot{\omega}_M ))\;.\end{equation*}where we recall that $p$
denotes the projection $Q_{Spin^c(n)} \rTo M$, and where 
$\rho_{\xi}$ acts on $\dot{\omega}_M$ via the isomorphism $\mf{so}(TM, g_{\xi}) \simeq 
\Cl_2(T^*M)$; the factor $1/2$ comes from the fact that $\nu$ is a $2:1$ 
covering.
Differentiating the relation~(\ref{nablaspin}) characterizing $\nabla_A^{W, \varphi_t \circ \xi}$, we get
$p^*(\dot{\nabla}^W_A) = (1/2)\; p^*(\rho_{\xi}( \dot{\omega}_M ) )$, yielding :
$$ \dot{\nabla}_A^W (s) =   \frac{1}{2} \rho_{\xi}( \dot{\omega}_M ) \;.$$
It remains now to determine the form $\dot{\omega}_M$; but this comes 
from the fact that $\varphi_t^* \omega_t$ is the connection form on 
$P_{SO(g_{\xi})}$ inducing the connection $\varphi_t^{-1} \nabla^{g_t} 
\varphi_t$ on $TM$: consequently its derivative $p^*(\dot{\omega}_M)$ 
 induces the form $\dot{\omega}_M = d (\varphi_t^{-1} \nabla^{g_t} 
\varphi_t) /dt |_{t=0}$, that is the form $ \dot{\omega}_M=\dot{\nabla}(s) - 
\nabla^{g_{\xi}} s \in A^1(M, \mathfrak{so}(TM, g_{\xi}))$, because the connection $\varphi_t^{-1} \nabla^{g_t} 
\varphi_t$ is compatible with the metric $g_{\xi}$ for all $t$.  \end{proof}
\subsection{Variation of the Levi-Civita connection}
In order to explicitely compute the variation $\dot{\nabla}^W_A(s)$ we only need the computation of the variation of the Levi-Civita connection 
$\dot{\nabla}(s)$ in the point $g_{\xi}$ along the symmetric tensor $s$. This is well known (see~\cite[Th. 1.174]{BesseEM}): 
\begin{pps}The variation of the Levi-Civita connection $\nabla^{g_\xi}$ on 
the tangent bundle $TM$ along the direction 
$s \in \sym(TM, g_{\xi})$ is the form $\dot{\nabla}(s) \in 
 A^1(M, \End(TM))$ given by: 
\begin{equation}\label{vlc} g_{\xi}(\dot{\nabla}(s)_X Y, Z) =   g_{\xi}( (\nabla^{g_{\xi}}_X s)Y,Z) -
g_{\xi}( (\nabla^{g_{\xi}}_Z s)Y,X) + g_{\xi}( (\nabla^{g_{\xi}}_Y s)Z,X)
 \;.\end{equation}
 \end{pps}
We now express formula (\ref{vlc}) in terms of a local orthonormal frame 
$e_i \in TM$, $i=1, \dots,n$ for the metric $g_{\xi}$. Let $e^i$ its dual frame. Let $\tau_{ij}^k$ and $c_{ij}^k$ the components of the tensors $\dot{\nabla} s$ and $\nabla^{g_{\xi}} s$, respectively,  with respect to 
the frame $e_i$: 
$$ \tau_{ij}^k = g_{\xi}(\dot{\nabla}(s)_{e_i} e_j, e_k) \qquad c^k_{ij}= 
g_{\xi}(\nabla^{g_{\xi}}_{e_i} s (e_j), e_k) \;.$$The tensor $c_{ij}^k$ is symmetric in $j$,$k$ because the Levi-Civita connection $\nabla^{g_{\xi}}$ preserves the bundle of symmetric endomorphisms $\sym(TM,g_{\xi})$ for the metric $g_{\xi}$. 
In the frame $e_i$, formula~(\ref{vlc}) reads:
$$ \tau_{ij}^k =  c_{ij}^k
-c^i_{kj} + c^{i}_{jk}  \;.$$The tensor $\tau_{ij}^k$ is symmetric in $i,j$. 
The components of the tensor $\dot{\omega}_M = \dot{\nabla}(s) - \nabla^{g_{\xi}} s$ are:
$$ \dot{\omega}_{ij}^k =  c^{i}_{jk}-c_{kj}^i \;.$$The tensor 
$\dot{\omega}^k_{ij}$ is skew-symmetric in $j$ and $k$, and hence belongs to 
$A^1(M, \mathfrak{so}(TM, g_{\xi}))$, as expected. 
\subsection{Variation of the Dirac operator}
We can now continue the computation began in the introduction of this section. 
Let us compute $\rho_{\xi} \circ \dot{\nabla}^W_A \varphi$ for a spinor $\varphi 
\in \Gamma(W)$. In the chosen local orthonormal frame $e_i$, we denote with $E^k_j$ the endomorphism $e^j \tens e_k -e^k \tens e_j$ in $\mathfrak{so}(TM, g_{\xi})$. We know that: 
$$ 
\dot{\omega}_M= \sum_{ijk} \dot{\omega}^k_{ij} e^i \tens e^j \tens e_k = \frac{1}{2}
\sum_{ijk} \dot{\omega}^k_{ij} e^i \tens E^j_k
\;,
$$therefore:
$$ \dot{\nabla}_A^W \varphi = \frac{\rho_{\xi}( \dot{\omega_M})}{2} \varphi = 
\frac{1}{4} \sum_{ijk} \dot{\omega}^k_{ij} e^i \tens \rho_{\xi}(e^j e^k) \varphi 
\;.$$
Consequently the term $\rho_{\xi} \circ \dot{\nabla}^W_A \varphi$ is: 
$$ \rho_{\xi} \circ \dot{\nabla}^W_A \varphi = 
\frac{1}{4} \sum_{ijk} \dot{\omega}^k_{ij}  \rho_{\xi}(e^i e^j e^k) \varphi \;.$$
Recalling now that $\dot{\omega}_{ij}^k = \tau_{ij}^k - c^k_{ij}$,
that 
$\tau^k_{ij}$ is symmetric in $i,j$ and that $c^k_{ij}$ is symmetric in $j,k$:
\begin{align*} \rho_{\xi} \circ \dot{\nabla}^W_A \varphi = \;&
-\frac{1}{4}\sum_{ij} \tau^j_{ii} \rho_{\xi}(e^j)  \varphi +
\frac{1}{4} 
\sum_{ij}c^i_{ji} \rho_{\xi}(e^j) \varphi \;. \end{align*}
Recalling the definition of $\tau_{ij}^k$ and that of $c_{ij}^k$, we have that 
$ \sum_{ij}\tau_{ii}^j e^j = 2\di s  - d \tr s$ and $\sum_{ij}
c_{ji}^i e^j = 
d \tr s$. 
Hence we get: 
$$ \rho_{\xi} \circ \dot{\nabla}^W_A \varphi =
 - \frac{1}{2}\rho_{\xi}(\di s -d\tr s) \varphi \;.$$
We proved the theorem: 
\begin{theorem}
The variation of the family $\mathpzc{D}$ of Dirac operators: $\Xi \ni \xi \rTo D_A^{\xi} \in \mathrm{Diff}^1(W)$ in the point $\xi$ along the horizontal direction 
corresponding to the variation of the metric $g_{\xi}$ in the direction $s \in \sym(TM, g_{\xi})$ is the first order differential operator given by:
\begin{equation}\label{Dirac} \ddt D_A ^{\xi_t} \rest_{t=0} = -\rho_{\xi} \circ s^* \circ \nabla^{W, \xi}_A - \frac{1}{2}
\rho_{\xi}(\di s -d\tr s) \;.
 \end{equation}
\end{theorem}The result agrees with the one obtained by Bourguignon and Gauduchon (see \cite[Th. 21]{BourguignonGauduchon1992}).
\section{Variation of the Seiberg-Witten equations}\label{sec:vswe}
The aim of this section is to compute the full differential of the 
\emph{universal Seiberg-Witten 
functional} 
$\mbb{F}_{\Met(M)}: \mc{A}_{U(1)}(L) \times \Gamma(W_+) \times \Xi \rTo 
\Gamma(W_-) \times i
\mathfrak{su}(W_+)$
on a solution $(A, \psi, \xi)$. 
We will denote with $\mbb{F}_1$ and $\mbb{F}_2$ the components of $\mbb{F}_{\Met(M)}$
with values in $\Gamma(W_-)$ and $i\mf{su}(W_+)$, respectively.
We will  use the splitting $T\Xi \simeq V \oplus H$ 
defined by the natural connection on $\Xi$.
The most interesting part of this computation is the one dealing with the variation 
of $\mbb{F}_{\Met(M)}$ along the horizontal direction, that is, the perturbation of the metric; the difficult point was the 
variation of the Dirac operator, treated in the previous section.

We  study here the variation of the equation (\ref{psw2}) with respect to the 
metric. Since the Clifford multiplication on $\Lambda^2T^*M$ transforms 
as $\rho_{\phi^{-1}\circ \xi} = \rho_{\xi} \circ (\Lambda^2 \phi^*)^{-1}$,
we have to differentiate the map: 
\begin{diagram}[height=0.6cm] \mbb{F}_2(A, \psi, \sigma_{\xi}(-)):  
\Sym^+(TM, \phi^* g_{\xi}) & \rTo & i
\mathfrak{su}(W_+) \\
\qquad \qquad \qquad \phi & \rMapsto & (\rho_{\xi} \circ (\Lambda^2 
\phi^*)^{-1})( F^{+, \phi^*g_{\xi}}_A ) -[\psi^* \tens \psi]_0
\end{diagram}at the identity.
The map $\phi$ is an orientation-preserving isometry between $(TM, \phi^*g)$ and 
$(TM,g)$, hence the Hodge star for the metric $\phi^* g_{\xi}$ 
can be expressed as:
$ *_{\phi^* g_{\xi}} = \Lambda^2 \phi^* \circ *_{g_{\xi}} \circ 
(\Lambda^2 \phi^*)^{-1}$.  
Consequently: 
$$ F^{+, \phi^* g_{\xi}}_A = \left( \frac{\Lambda^2 \phi^* \circ *_{g_{\xi}} \circ 
(\Lambda^2 \phi^*)^{-1} +1}{2}\right)F_A = \Lambda^2 \phi^* 
\left( \frac{*_{g_{\xi}}  
 +1}{2}\right)(\Lambda^2 \phi^*)^{-1} F_A \;.$$
Therefore, denoting $P^{+,g_{\xi}}$ the projection onto 
self-dual $2$-forms for the 
metric $g_{\xi}$, we get: 
$$ (\rho_{\xi} \circ (\Lambda^2 
\phi^*)^{-1})( F^{+,\phi^* g_{\xi}}_A ) = \rho_{\xi} (P^{+,g_{\xi}} \circ 
(\Lambda^2 \phi^*)^{-1} F_A) \;.$$
Given a path of metrics in the direction $s \in \sym(TM, g_{\xi})$, $g_t = \phi^*_t g_{\xi}$, $\phi_t = 1 +ts$, and differentiating in $t=0$ we get: 
$$ \de{\mbb{F}_2}{\phi} (A, \psi, \id)(s)= - \rho_{\xi}(P^{+, g_{\xi}} i(s^*) F_A \; ) \;,$$
where $i(s^*)$ is the derivation of degree $0$ on 
$\Lambda^* T^*M$ that coincides with $s^*$ on $T^*M$. In order to better 
understand the term $P^{+, g_{\xi}} i(s^*) F_A$ consider the splitting of symmetric endomorphisms:
$$ \sym(TM, g_{\xi}) \simeq \sym_0(TM,g_{\xi}) \oplus C^{\infty}(M, 
\mbb{R}) \cdot \id_{TM}$$ 
in traceless ones and scalar ones. 
It is now well known that the bundle $\sym(TM, g_{\xi})$ embeds in $\End(\Lambda^2T^*M)$ via the morphism $s \rMapsto i(s^*)$. According to the decomposition $\Lambda^2T^*M \simeq \Lambda^2_+ T^*M \oplus 
\Lambda^2_- T^*M$ and indicating with $s_0$ the traceless part of $s$, we can express $i(s^*)$ as: $$ i(s^*) = \matrice{cc}{\tr s & P^{+, g_{\xi}} i(s_0^*) \trest_{\Lambda^2_- T^*M} \\
P^{-, g_{\xi}} i(s_0^*) \trest_{\Lambda^2_+ T^*M} & \tr s } \;. $$Hence
 there remain induced
isomorphisms 
\begin{equation}\label{eq:delta}
\delta_{\pm}: \sym_0(TM, g_{\xi}) \rTo \Hom(\Lambda ^2_{\pm}T^*M, \Lambda^2_{\mp}T^*M)
\end{equation}and an isomorphism between 
scalar endomorphisms of $T^*M$ and  homoteties of $\Lambda^2T^*M$.
Therefore 
$ P^{+, g_{\xi}} i(s^*) F_A = (\tr s )F_A^+ + \delta_-(s_0)F_A^- $ and
$$  \de{\mbb{F}_2}{\phi} (A, \psi, \id)(s)= -(\tr s)\rho_{\xi}(F_A^+) - \rho_{\xi}(\delta_-(s_0)F_A^-) \; $$

The differential of the functional $\mbb{F}_{\Met(M)}$ with respect to the connection and the 
spinors components is well known (see \cite{MorganSWEAFMT}): we have, for variations 
$\tau \in iA^1(M)$, $\varphi \in \Gamma(W_+)$:
$$ \de{\mbb{F}_{\Met(M)}}{(A, \psi)}(A, \psi, \xi)(\tau, \varphi)= 
\left( 
\begin{array}{c}
{\displaystyle \frac{1}{2}} \rho_{\xi}(\tau) \psi + D_A^{\xi } \varphi \\
d^{+}\tau -[\varphi^* \tens \psi + \psi^* \tens \varphi]_0 \end{array}
\right) $$
To compute the full differential of $\mbb{F}_{\Met(M)}$ it remains to compute its
 variation along vertical direction of $\Xi$, that is, along the fibers. 
Since the group $\Aut(Q_{Spin^c(4)})$ acts on the configuration space $\mc{C} \times 
\Xi$ preserving the solutions and since its action on $\Xi$ is transitive 
on the 
connected component of the fibers, 
 no contribution to the 
transversality can be obtained in this way, because the component of the 
differential we get is  a linear 
combination of the other components. 
\section{The question of transversality}
In this section we set up the transversality problem. Our final project is to 
prove that the universal Seiberg-Witten moduli space $\mc{M}_{\Met(M)}$  is  
smooth, at least at its irreducible points. This can be achieved with standard 
methods via the implicit function theorem applied to adequate
Banach manifolds, once we know that the defining equations of $\mc{M}_{\Met(M)}$ inside 
$(\mc{C} \times \Xi)/\mc{G}$ are transversal at irreducible monopoles. 
In order to proceed in such a way we need, as usual, 
to complete our till now Fr\'echet manifolds to Banach~ones.
\subsection{Sobolev completions}\label{sbvc}
Let $\mc{C}^2_p = \Gamma(W_+)^2_{p} \times \mc{A}_{U(1)}(L)^2_p$ 
and $\mc{D}^2_{p-1} = 
i \mathfrak{su}(W_+)^{2}_{p-1} \times \Gamma(W_ -)^2_{p-1}$ be the completions of $\mc{C}$ and $\mc{D}$ in Sobolev norms $|| \; \; ||_{2,p}$ and $|| \; \; ||_{2,p-1}$, respectively, so that they become a Hilbert affine space and a Hilbert vector space, respectively. Consider also the Banach completions 
$\mc{C}^l$ and $\mc{D}^l$  of 
$\mc{C}$ and $\mc{D}$, respectively, in norm $C^l$, $l \in \mbb{N}$.
The space of metrics $\Met(M)$ can be completed to a Banach manifold
considering the space of $C^r$-metrics $\Met(M)^r$. We suppose that the Fr\'echet submanifold $T \subseteq \Met(M)$ we are considering admits a completion\footnote{This is always the case if $T$ is a space of global sections of a fiber subbundle of $Met(M)$} to a Banach splitting
submanifold $T^r$ of $\Met^r(M)$.
Complete now $\Xi$ with $\mu$-equivariant morphisms of class~$C^r$:  
$ \Xi^r := \Mor_{\mu}^r(Q_{Spin^c(4)}, P_{GL_+(4)}) \; ;$
the space $\Xi^r$ becomes then a Banach principal 
bundle with structural group $\Aut^r(Q_{SO(n)})$ (the $C^r$-gauge group of 
$Q_{SO(n)}$) over the space of $C^r$-metrics $\Met^r(M)$; the natural
connection defined in subsection \ref{subsect:natconnection} extends
to this setting.
Now take $\Xi^r_T$ to be the $\Aut^r(Q_{SO(n)})$-Banach principal bundle on $T^r$ given by the restriction of $\Xi^r$ to $T^r$. 
 We will always suppose $r>>p>>0$. We will complete as well the 
gauge group $\mc{G}$ to $\mc{G}^2_{p+1}$, in Sobolev norm $|| \; \; ||_{2,p+1}$ 
in order to have a Banach-Lie group acting continuously on 
$\mc{C}^2_p$ and $\mc{D}^2_{p-1}$.
Denote with $\mc{C}^{2,*}_p$ the irreducible unparametrized configurations, 
that is, couples $(A, \psi)$ in $\mc{C}^2_p$ such that $\psi \neq 0$: such 
couples have trivial $\mc{G}^{2}_{p+1}$-stabilizer;
denote moreover with $\mc{B}^{2}_p$ the quotient $\mc{B}^{2}_p := \mc{C}^2_p / 
\mc{G}^2_{p+1}$ and with $\mc{B}^{2,*}_p := \mc{C}^{2,*}_p / \mc{G}^2_{p+1}$; 
the latter is a Hilbert manifold.
Let now $(\mf{C}_T)^{2,r}_p := \mc{C}^2_p \times \Xi_T^r$ be the space of parametrized configurations, completed in Sobolev and $C^r$ norm, and $(\mf{C}_T^*)^{2,r}_p = \mc{C}^{2,*}_p \times \Xi_T^r
$ the irreducible ones.
The quotient 
$(\mathfrak{B}_T^*)^{2,r}_p:= (\mf{C}^*_T)^{2,r}_{p} / \mc{G}^2_{p+1}$ is 
isomorphic to
$\mc{B}^{2,*}_p \times \Xi^r_T$ and hence a Banach manifold.
 The functional $\mbb{F}_T$ extends to a $\mc{G}^2_{p+1}$-equivariant map: 
$$ (\mbb{F}_T)^{2,r}_{p}:  (\mf{C}_T)^{2,r}_p  \rTo \mc{D}^2_{p-1} \;.$$
We indicate with $(\mc{M}_T)^{2,r}_p := 
Z( (\mbb{F}_T)^{2,r}_p) / \mc{G}_{p+1}$ the parametrized Seiberg-Witten moduli space and with
 $(\mc{M}_T^*)^{2,r}_p = (\mc{M}_T)^{2,r}_p \cap (\mathfrak{B}^*_T)^{2,r}_p$ the parametrized moduli space of irreducible monopoles.
\begin{remark}\label{rmk:geq}
Let $(A, \psi, \xi) \in Z((\mbb{F}_T)^{2,r}_p)$ be a solution to the
parametrized Seiberg-Witten equations. Then $(A, \psi, \xi)$ is
$\mc{G}^2_{p+1}$-equivalent to a solution $(A^{\prime}, \psi^{\prime},
\xi)$, 
with $(A^{\prime}, \psi^{\prime}) \in 
\mc{C}^{r-3}=\mc{A}_{U(1)}(L)^{r-3} \times \Gamma(W_+)^{r-3}$ (of
class $C^{r-3}$)\footnote{This 
can be proved with
a slight modification (with $C^r$ regularity) of the proof given in 
\cite[Proposition 6.19]{TelemancourseDEA}. The proof is based on two statements. The first that $(A, \psi)$, 
as a solution of unparametrized Seiberg-Witten equations with fixed 
$Spin^c$-structure $\xi$, $(A, \psi)$ is $\mc{G}^2_{p+1}$-equivalent to $(A^{\prime \prime}, \psi^{\prime \prime})$ with $A^{\prime \prime}$ in $A_0$-gauge, for a fixed 
$C^{\infty}$-connection $A_0$, 
that is $d^*(A^{\prime \prime}-A_0)=0$. This proof \cite[Proposition
6.10]{TelemancourseDEA} is still valid with a $C^r$-metric. 
Then, that the
Seiberg-Witten equations (for fixed $\xi$), restricted to 
solutions in $A_0$-gauge, 
become the following nonlinear elliptic equations in $\tau= A^{\prime \prime} -A_0$ and $\psi$: 
\begin{gather}
D_{A_0} \psi = -\frac{1}{2} \rho_{\xi}(\tau) \psi \\
\left( \begin{array}{c} d^+ \\ d^*
\end{array} \right) \tau = \left( \begin{array}{c} [\psi^* \tens \psi]_0 -F_{A_0}^+ \\
 0
\end{array} \right) 
\end{gather}
By a bootstrapping argument based on 
Sobolev multiplication and elliptic regularity (see \cite[Appendix K,
Theorem 40]{BesseEM})
(the left hand term is a first order elliptic differential operator with 
$C^{r-1}$-coefficients) we get that the solution 
$(A^{\prime \prime}, \psi)$ is in $\mc{C}^2_{r}$ and hence in $\mc{C}^{r-3}=\mc{A}_{U(1)}(L)^{r-3} \times \Gamma(W_{+})^{r-3}$.}.
Consequently, 
if $\xi$ has $C^{\infty}$-regularity, up to the $\mc{G}^2_{p+1}$-action, we can always 
suppose that $(A, \psi)$ have $C^{\infty}$-regularity.
\end{remark}

Consider now the trivial Banach vector bundle: 
$ \mc{E}_T : (\mathfrak{C}^*_T)^{2,r}_{p} \times \mc{D}^2_{p-1}$ over the irreducible parametrized configurations $(\mathfrak{C}^*_T)^{2,r}_{p}$
: it is naturally equipped with a $\Aut^r(Q_{Spin^c})$-action. Denote with 
$\Gamma^r$ the image $\Gamma^r := \pim (\Aut^r(Q_{Spin^c}) \rTo \Aut^r(Q_{SO(4)}
))$. 
The restriction
of the functional $(\mbb{F}_T)^{2,r}_{p}$ 
to $(\mathfrak{C}^*_T)^{2,r}_{p}$
descends to the $\mc{G}^2_{p+1}$-quotients 
to give a $\Gamma^r$-equivariant section $\Psi_T$ of the $\Gamma^r$-equivariant quotient bundle 
$ \mathfrak{E}_T:= \mc{E}_T/ \mc{G}^2_{p+1}$:
$$ \Psi_T: (\mathfrak{B}^*_T)^{2,r}_{p}
 \rTo \mathfrak{E}_T\;.$$The section $\Psi_T$ is a Fredholm map between Banach manifolds;
its zero set $Z(\Psi_T)$ is exactly the parametrized moduli space
$(\mc{M}^*_T)^{2,r}_p$ of irreducible monopoles. 

\subsection{Remarks on the transversality statement}
A metric $g \in \Met(M)^r$ is said $c$-good (for the fundamental class
$c \in H^2(M,\mbb{Z})$) if the projection of $c$ onto the 
$g$-harmonic self-dual classes $\mc{H}^2_+$ is non zero. 
We denote with $\Met_{\textrm{$c$-good}}^r(M)$ the set 
of such metrics: 
it is an open set of 
$\Met^r(M)$, complementary of a closed of codimension $b_+(M)$. Let
now $b_+(M)>0$ and 
suppose that $\Met(M)^r_{c-\rm{good}} \cap T$ is a dense open set of $T$.
Let $\Xi^{r,**}_T$ the subset of $Spin^c$-structures in $\Xi^r_T$
projecting onto 
$c$-good metrics; 
set $\Xi^{r,**}:= \Xi^{r,**}_{\Met(M)}$. 
Let finally $(\mc{M}_{T}^{**})^{2,r}_p:=Z(\Psi_{T}) \cap (\mc{B}^2_p \times 
\Xi^{r,**}_T)$. In the following we will always denote with
$\Met(M)_{\textrm{c-good}}$, $\Xi^{**}$, $\Xi^{**}_T$, $\mc{M}^{*}_T$, 
$\mc{M}^{**}_T$ the corresponding spaces of objects of class $C^{\infty}$.
The projection $\pi_T \colon (\mc{M}_{T}^{**})^{2,r}_p \rTo \Xi^{r,**}_T$ is now a
surjective Fredholm map with \emph{compact fibers}, since no reducible monopole is allowed over a $c$-good metric. Denote with $(\mc{M}_{\xi})^{2,r}_p$ its fiber over $\xi \in \Xi^{r,**}_T$.
A satisfying result would be the following:
\begin{stat}\label{stat}The intrinsic differential $D_x\Psi_{\Met(M)}$ is surjective at all points
$x \in (\mc{M}^*_{\Met(M)})^{2,r}_p$. Consequently, the universal
moduli space $(\mc{M}^*_{\Met(M)})^{2,r}_p$ of irreducible Seiberg-Witten monopoles is a smooth Banach manifold.
\end{stat}
\begin{remark}\label{cinfty}
We discuss here the consequences of transversality. If statement~\ref{stat} is true, a standard application of Sard-Smale theorem \cite{Smale1965} implies that, 
\emph{for 
$\xi$ in a dense open set of $\Xi^{r, **}$, 
the Seiberg-Witten 
moduli space $(\mc{M}_{\xi})^{2,r}_p$ is smooth of the expected dimension. Consequently, 
this would be true for generic $\xi$ of class~$C^{\infty}$}. 
These arguments can be adapted in the following case, useful in the applications to k\"ahlerian monopoles. 
Let $S$ a 
Fr\'echet submanifold of $\Met(M)$, embedded in $T$, and let now
$\mc{M}_S$ be the parametrized Seiberg-Witten moduli space\footnote{As explained above, here we consider the moduli space of $C^{\infty}$-objects, introduced in section \ref{2.5}; by remark \ref{rmk:geq} it can be seen as embedded
 in $(\mc{M}_{\Met(M)})^{2,r}_p$, as being the inverse image of smooth elements
 $\Xi_S$ for the projection $\pi_{\Met(M)}$}
over $S$. Suppose that the intrinsic differential 
$D_x \Psi_T$ is surjective at all points of $\mc{M}_S^{**}$; in this case 
the parametrized moduli space $(\mc{M}_T)^{2,r}_p$and the universal
moduli space $(\mc{M}_{\Met(M)})^{2,r}_p$ are smooth at points in $\mc{M}_S^{**}$.
Fix now  a $Spin^c$-structure $\xi \in \Xi_S^{**}$ (of class $C^{\infty}$). By compactness of $(\mc{M}_{\xi})^{2,r}_p \simeq 
\mc{M}_{\xi}= \pi^{-1}_T(\xi)$, there exists an open neighbourhood $V_{\xi}$ 
of $\xi$ in 
$(\Xi_T^{**})^{2,r}_p$ such that $D_y \Psi_T$ is surjective for all $y \in 
\pi^{-1}_T(V_{\xi})$: hence
$\pi_T^{-1}(V_{\xi})$ is 
a smooth Banach manifold and the projection $\pi_T^{-1}(V_{\xi}) \rTo V_{\xi}$
is a Fredholm map with compact fibers. We can now apply to this map a 
standard argument using Sard-Smale theorem to obtain that for a generic 
$\xi^{\prime} \in V_{\xi}$ the fiber $(\mc{M}_{\xi^{\prime}})^{2,r}_p$ is smooth of the expected dimension; hence for a generic $\xi^{\prime} \in V_{\xi} \cap \Xi_T$ of class $C^{\infty}$ the
moduli space $\mc{M}_{\xi^{\prime}}$ is smooth of the expected dimension. The same can be said for an adequate neighbourhood $W_{\xi}$ of $\xi$ in~$\Xi^{r,**}$.
\end{remark}

We will now discuss the condition that the intrinsic differential $D_x\Psi_T$ 
is surjective at a point $x \in (\mc{M}_T^*)^{2,r}_p$. 
We recall that a local slice for the action of $\mc{G}^2_{p+1}$ in $\mc{C}^2_p$ 
at the point $x$ 
is a smooth splitting Hilbert submanifold $\mc{S}_x$ of a 
neighbourhood of $x$, invariant under the stabilizer $\Stab(x)$, 
such that the natural map: 
$$ \mc{S}_x \times _{\Stab(x)} \mc{G}^2_{p+1} \rTo \mc{C}^2_{p}$$is a diffeomorphism 
onto a neighbourhood of the orbit through $x$. Such a slice exists
\cite[Lemma 4.5.5]{MorganSWEAFMT} and is built, in a neighbourhood of a point $(A, \psi)$, 
as $$\mc{S}_{(A, \psi)}=\{ (A^{\prime}, \phi) \in \mc{C}^2_p \, | \,  (D_{(A,\psi)}
\Upsilon_{A, \psi})^*(A^{\prime}-A, \phi-\psi) =0 \; , || A^{\prime}-A
||^2_p < \epsilon\, , \, || \phi - \psi ||^2_p < \epsilon
\}$$for a sufficiently small $\epsilon$, where $\Upsilon_{A, \psi}:
\mc{G}^2_{p+1} \rTo \mc{C}^2_p$ is the given by group action at $(A,
\psi)$ and where $(D_{(A, \psi)} \Upsilon_{(A, \psi)})^*$ is the
formal adjoint of the differential $D_{(A, \psi)} \Upsilon_{(A, \psi)}$.
 As a consequence $T_{(A, \psi)} \mc{S}_{(A, \psi)} \simeq \ker  (D_{(A, \psi)} \Upsilon_{A, \psi})^*$. A slice for the 
action of $\mc{G}^2_{p+1}$ on 
$(\mf{C}^*_T)^{2,r}_p$
at a point 
$(A, \psi, \xi)$ is then given by $\mc{S}_{(A, \psi, \xi)}:= \mc{S}_{(A, \psi)}\times U_{\xi}$ where $U_{\xi}$ 
is a small neighbourhood of $\xi$ in $\Xi^r_T$. 
The slice $\mc{S}_{(A, \psi, \xi)}$ provides a local 
model for $(\mathfrak{B}^*_T)^{2,r}_p$ at an irreducible point 
$([A, \psi], \xi)$. 
\begin{pps}\label{ciana}
The section $\Psi_T$ is transversal to the zero section at the point 
$x= ([A, \psi], \xi) \in (\mc{M}_T^*)^{2,r}_p$ if and only if the functional 
$(\mbb{F}_T)^{2,r}_{p}$ is transversal to $0$ 
at the point $(A, \psi, \xi) \in (\mathfrak{C}^*_T)^{2,r}_p$. 
\end{pps}\begin{proof}The section $\Psi_T$ can be written locally on $\mc{S}_{(A, \psi, \xi)}$ as: 
\begin{diagram}[height=0.5cm] 
\mc{S}_{(A, \psi)} \times U_{\xi} & \rTo & \mc{D}^2_{p-1} \\
((A^{\prime}, \psi^{\prime}), \xi^{\prime}) & \rMapsto & 
(\mbb{F}_T)^{2,r}_p(A^{\prime}, \psi^{\prime}, \xi^{\prime})
\end{diagram}
Since $T_{(A, \psi)} \mc{S}_{(A, \psi)}   \simeq \ker  (D_{(A, \psi)} \Upsilon_{A, \psi})^* = (\pim
 D_{(A, \psi)} \Upsilon_{A, \psi} )^{\perp}$, 
we have that $D _x\Psi_T$ coincides with the restriction $D_{(A, \psi, \xi)}
(\mbb{F}_T)^{2,r}_p \trest_{(\pim 
 D_{(A, \psi)} \Upsilon_{A, \psi} )^{\perp} \oplus T_{\xi} \Xi^r_T}
$. Now, since 
$\pim
 D_{(A, \psi)} \Upsilon_{A, \psi} \oplus T_{\xi} \Xi^r_T \subseteq \ker D_{(A, \psi, \xi)}
(\mbb{F}_T)^{2,r}_p $, we can conclude. \end{proof}
\begin{remark}\label{postequi}Since it is $\Gamma^r$-equivariant, the section $\Psi_T$ is 
transversal to zero in a point $x= ([A, \psi], \xi)$ if and only if it is 
transversal to zero in a point $xf= ([f^*A, \psi f], \xi f)$, for 
$f \in \Gamma^r$.
This means that the smoothness of a standard 
Seiberg-Witten moduli space $(\mc{M}_{\xi})^{2,r}_p$, $\xi \in \Xi^r_T$ does not depend 
on the particular $Spin^c$-structure $\xi$ (chosen in a fixed connected 
component of $\Xi^r_T$) but \emph{only} on the metric $g_{\xi}$ compatible with 
$\xi$. \end{remark}
\begin{remark}\label{postequi2}
Let $b_+(M)>0$. By remark \ref{postequi} and since 
the torsion subgroup $\mf{t}$, counting
connected components of $\Xi/\Gamma$, is finite, 
if statement \ref{stat} is true, then for a generic 
$C^{\infty}$ metric $g \in \Met(M)$, the standard 
Seiberg-Witten moduli $\mc{M}_{\xi}$ is smooth of the expected dimension for 
any $Spin^c$-structure $\xi$,  
compatible with $g$. 
This means that, even there is no satisfactory way to parametrize 
Seiberg-Witten equations and moduli spaces just with metrics $\Met(M)$,
because, by remark \ref{curvrmk}, there are no parallel sections $\Met(M) \rTo \Xi$,
 the transversality statement depends \emph{only} on the metric 
chosen and not on the particular $Spin^c$-structure compatible with the metric.
\end{remark}
\subsection{The adjoint operator}\label{obstransv}
In this subsection we express the obstruction to the transversality of the section $\Psi_T$ in terms of the formal 
adjoint of the differential of the functional $\mbb{F}_T$. Let 
$([A, \psi], \xi) \in Z(\Psi_T)$. By remark \ref{rmk:geq} we can suppose that 
$(A, \psi)$ are of class $C^{r-3}$. By proposition \ref{ciana} the obstruction to the transversality of $\Psi_T$ at the point $([A, \psi], \xi)$ is given by the cokernel of the first order differential operator with $C^{r-3}$ coefficients: 
$$D_{(A, \psi, \xi)}
(\mbb{F}_T)^{2,r}_p: T_{(A, \psi)} \mc{C}^2_p \oplus T_{\xi}\Xi_T^r
\rTo \mc{D}^2_{p-1} \;.$$The operator $D_{(A, \psi, \xi)}
(\mbb{F}_T)^{2,r}_p$ is the partial Sobolev completion of the first order differential operator (with $C^{r-3}$ coefficients), given by the differential 
of the Seiberg-Witten functional $\mbb{F}_T$, defined on parametrized configurations of class $C^{r-3}$ and $C^r$:
$$ D_{(A, \psi, \xi)}
\mbb{F}_T \colon T_{(A, \psi)} \mc{C}^{r-3} \oplus T_{\xi} \Xi_T^{r} \rTo 
\mc{D}^{r-4} \;.$$
We indicate with $D_{(A, \psi)}F^{\xi}$ and with $P$ the first and 
second component,
respectively. 
The component 
$ D_{(A, \psi)}F^{\xi}$ 
is the differential of the unparametrized Seiberg-Witten functional $F^{\xi}$, relative to the $Spin^c$-structure $\xi$: it is well known \cite{MorganSWEAFMT} that it is underdetermined elliptic, in this case with $C^{r-3}$ coefficients; hence its Sobolev extension 
$(D_{(A, \psi)} F^{\xi})^2_p : T_{(A, \psi)} \mc{C}^2_p \rTo \mc{D}^2_{p-1}$ has closed
image of finite codimension. 
Since the operator $D_{(A, \psi, \xi)}
(\mbb{F}_T)^{2,r}_p$ can be written as the sum
$ D_{(A, \psi, \xi)}
(\mbb{F}_T)^{2,r}_p = (D_{(A, \psi)} F^{\xi})^2_p +P$, 
 it follows\footnote{This follows from the following fact \cite[Lemma 6.36]{TelemancourseDEA}. Let 
$E$ a Banach vector space with a continuous scalar product $<\cdot,\cdot> : E \times E \rTo \mbb{R}$ and let $F \subset E$ a closed subspace of finite codimension such that its orthogonal $F^{\perp}$ is a topological supplementary  of $F$ in $E$. If $F^{\prime}$ is a subspace of $E$ containing $F$, then $F^{\prime}$ is closed of finite codimensionand its orthogonal $(F^{\prime})^{\perp}$ is a topological supplementary of $F^{\prime}$ in $E$.} that 
$D_{(A, \psi, \xi)}
\mbb{F}_T$ is underdetermined elliptic and that $D_{(A, \psi, \xi)}
(\mbb{F}_T)^{2,r}_p$ has closed image of finite
codimension. By elliptic regularity, we  have firstly that $\coker D_{(A, \psi, \xi)}
(\mbb{F}_T)^{2,r}_p \subseteq \coker (D_{(A, \psi)} F^{\xi})^2_p \subseteq \mc{D}^{r-4}$; secondly that $\coker D_{(A, \psi, \xi)}
(\mbb{F}_T)^{2,r}_p$
can be identified with the $L^2$-orthogonal $(\pim  D_{(A, \psi, \xi)}
(\mbb{F}_T)^{2,r}_p)^{\perp}$; the latter coincides with the kernel 
of the formal adjoint of $D_{(A, \psi, \xi)}
\mbb{F}_T$: 
\begin{equation}\label{ecci} \coker D_{(A, \psi, \xi)}
(\mbb{F}_T)^{2,r}_p \simeq (\pim (D_{(A, \psi, \xi)}
(\mbb{F}_T)^{2,r}_p)^{\perp} \simeq \ker (D_{(A, \psi, \xi)}
\mbb{F}_T)^* \trest_{\mc{D}^{r-4}} \;.\end{equation}
\begin{remark}\label{transvrmk1}For $\xi \in \Xi^r$ 
consider  the section 
$\sigma_{\xi}\colon \Met(M)^r \rTo \Xi^r$ passing through $\xi$ and defining the horizontal 
distribution $H_{\xi}$ on $\Xi^r$. 
Suppose now that $g_{\xi} \in
T^r$. Let $\sigma_{T, \xi}$ be the restriction 
$\sigma_{T, \xi}:=\sigma_{\xi}\trest_{T^r}$ of this 
section to the submanifold $T^r$. We denote with 
$\tilde{\mbb{F}}_T$ the composition: 
$$  \tilde{\mbb{F}}_T := \mbb{F}_T \circ (\id_{\mc{C}^{r-3}} 
\times \sigma_{T, \xi}) : 
\mc{C}^{r-3} \times T^r \rInto \mc{C}^{r-3} \times \Xi^r_T \rTo \mc{D}^{r-4}\;.$$
Since, as seen in the end of section \ref{sec:vswe}, 
there is no contribution to the 
transversality
coming from variations of the equations 
along vertical directions in $T_{\xi}\Xi_T$, we have that 
\begin{equation}\label{BBB}
 \coker D_{(A, \psi, \xi)} (\mbb{F}_T)^{2,r}_p  \simeq  \ker 
 (D_{(A, \psi, \xi)}  \mbb{F}_T)^* \trest_{\mc{D}^{r-4}} \simeq  
\ker (D_{(A, \psi, g_{\xi})}  \tilde{\mbb{F}}_T)^* \trest_{\mc{D}^{r-4}} \;.
\end{equation}
\end{remark}
\begin{remark}\label{transvrmk2}
Since $\tilde{\mbb{F}}_T$ factorizes in the composition
$\mc{C}^{r-3} \times  T^{r} \rInto 
\mc{C}^{r-3} \times \Met(M)^r \rTo^{\tilde{\mbb{F}}_{\Met(M)}} \mc{D}^{r-4}$, we have: 
\begin{equation*}(D_{(A, \psi, g_{\xi})} \tilde{\mbb{F}}_T)^* = 
(\id_{T_{(A, \psi)}\mc{C}^{r-3}} \oplus P_{T_{g_{\xi}}T^r}) \circ 
(D_{(A, \psi, g_{\xi})} \tilde{\mbb{F}}_{\Met(M)})^* \;,
\end{equation*}where  $P_{T_{g_{\xi}}T^r}$ denotes the orthogonal projection 
$T_{g_{\xi}}\Met(M)^r \rTo T_{g_{\xi}}T^r$. \end{remark}In what follows we will
denote more briefly by $\tilde{\mbb{F}}$ the functional
$\tilde{\mbb{F}}_{\Met(M)}$. It is clear that if $\xi$ is of class
$C^{\infty}$, we can drop the superscripts, considering spaces of
objects of class  $C^{\infty}$.
\paragraph{Computation of the adjoint operator.}
In the sequel we will always assume for simplicity's sake that the point $x=(A, \psi, \xi)$,  where the differential is computed, is such that 
$A, \psi, \xi$ are of class\footnote{This will be 
always the case in the applications; however what we will say holds in
all generality.} $C^{\infty}$. In this subsection we will compute the
formal adjoint of the differential of $\tilde{\mbb{F}}$ at the point
$(A, \psi, g_{\xi})$ 
\begin{equation*}\label{pardiff}
D_{(A, \psi, g_{\xi})} \tilde{\mbb{F}} : iA^1(M) \times \Gamma(W_+) \times \sym(TM,g_{\xi}) \rTo 
\Gamma(W_-) \times i A^2_+(M) \;, 
\end{equation*}given by:
$$ 
D_{(A, \psi, g_{\xi})} \tilde{\mbb{F}}(\tau, \varphi,  s) = \matrice{c}{\displaystyle 
\frac{1}{2} \rho_{\xi}(\tau) \psi + D_A^{\xi } \varphi - \rho_{\xi} \circ s^* \circ \nabla^{W, \xi}_A \psi - \frac{1}{2}\rho_{\xi}(\di s -d \tr s)\\
d^{+}\tau -[\varphi^* \tens \psi + \psi^* \tens \varphi]_0 - (\tr s)F_A^+ - \delta_{-}(s_0)F_A^- 
} \; .
$$The computation of the formal adjoint 
is mostly straightforward: we will just remark the less trivial steps and make clear some notations. 
\paragraph{\it $L^2$-norms.} We recall here the $L^2$-norms 
with respect to which we are going to compute its formal adjoint; we will always indicate with
 $(\cdot,\cdot)$ the real inner products and with $\langle \cdot,\cdot
 \rangle$ the hermitian ones. On the bundle $T^*M \tens i \mbb{R}$ the norm is the standard one
induced by the metric $g_{\xi}$.  On $T^*M^{\tens m}$ the
metric $g_{\xi}$ induces the inner product $( x_1 \tens \dots \tens x_m , y_1
\tens \dots \tens y_m ) = m!\prod_{i=0}^m (x_i, y_i)$; 
the decomposition $T^*M \tens T^*M \simeq S^2T^*M
\oplus \Lambda^2T^*M $ is then an orthogonal direct sum. 
We will take on $S^2T^*M$ and on $\Lambda ^2T^*M$ the metrics induced by
the metric on $T^*M \tens T^*M$.
In this way
$|| e_i \tens e_j ||^2 =2$,
$|| e_i e_j ||^2=1= ||e_i \w e_j||^2$, if $i \neq j$, otherwise $|| e_i \tens e_i||^2 = || e_i^2 ||^2 =2$. The metric induced by
$T^*M \tens T^*M$ on $\sym(TM,g_{\xi})$ is $(s,t) = 2 \tr(s t)$. 
The morphisms $\delta_{\pm}$ defined in (\ref{eq:delta})
are isometries if we take on $\Hom(\Lambda^2_{\pm}T^*M, \Lambda^2_{\mp}T^*M)$ the
metric $(u,v)= 1/2 \tr (uv^*) $. 
On $\Gamma(W_+)$ and on $\Gamma(W_-)$
we take the real part of the hermitian metric and on $\Hom(W_+,W_-)$
 the hermitian scalar product $\langle u,v \rangle=1/2 \;  \tr(uv^*)$, so that the
Clifford multiplication $\rho_{\xi}$ is  an isometry. Finally the real part
of the hermitian
metric on $\End(W)$, given by $\langle A,B\rangle = 1/4 \;  \tr (A
 B^*)$, induces an orthogonal direct sum $\mathfrak{u}(W) \simeq i\mbb{R}
 \oplus \mathfrak{su}(W)$. We put on $i\mathfrak{su}(W)$ the
 real inner product induced by the real inner product just defined on
 $\End(W)$, so that the isomorphism $\rho_{\xi} : \Lambda^2T^*M \rTo i
 \mathfrak{su}(W)$ is an isometry. 
The isomorphism $\rho_{\xi} : TM \tens \mbb{C} \rTo \Hom(W_+,W_-)$ allows us to identify elements in $\Hom(W_+, W_-)$ with complexified tangent vectors and to define a complex conjugation (and hence a real and imaginary part) for elements in $\Hom(W_+,W_-)$.
\paragraph{\it Computation of the adjoint operator.}
We express the adjoint operator 
$(D_{(A, \psi, g_{\xi})} \tilde{\mbb{F}} )^*$ in terms of variables 
$(\chi, \theta) \in \Gamma(W_-) \times iA^2_+(M)$. In the rest of the article 
we will identify symmetric $2$-tensors  $S^2T^*M$ with symmetric endomorphisms $\sym(TM, g_{\xi})$ by means of the metric $g_{\xi}$.
\begin{enumerate}\item
To compute the adjoint of the map 
$j_{\psi} \colon A^1(M,\mbb{C}) \rTo 
\Gamma(W_-)$, given by $\sigma \rMapsto \rho_{\xi}(\sigma) \psi$, 
remark that, for $\chi \in \Gamma(W_-)$, we have:
\begin{equation}\label{adj}
\langle \rho_{\xi}(\sigma)\psi, \chi\rangle  =  \tr [\rho_{\xi}(\sigma) \circ (\psi^* \tens \chi )^*] = 2 \langle \rho_{\xi}(\sigma ), \psi^* \tens \chi 
\rangle_{\Hom(W_+,W_-)} 
=  2 \langle \sigma , \psi^* \tens \chi 
\rangle_{T^*M \tens \mbb{C}} \;.
\end{equation}Therefore the \emph{hermitian adjoint} of $j_{\psi}$ is the map 
$\chi \rMapsto 2 \: \psi^* \tens \chi $. 

\item The adjoint of the map $q_{\psi}(\varphi)=[\varphi^* \tens \psi + \psi^* \tens \varphi]_0$ is the operator $q^*_{\psi }:iA^2_+(M) \rTo \Gamma(W_+)$ given by: 
$q^*_{\psi }(\theta) = 1/2 \rho_{\xi}(\theta)\psi$. 
This can be proved firstly showing that, with the taken norms:
\begin{equation}\label{ide}\langle\rho_{\xi}(\theta), [\phi^* \tens \phi]_0 \rangle = \frac{1}{4}
\langle \rho_{\xi}(\theta) \phi, \phi \rangle \qquad \forall \; \theta \in iA^2_+(M) \quad \forall \phi \in  \Gamma(W_+)  \end{equation}and, secondly, differentiating the identity 
(\ref{ide}) with respect to $\phi$ and identifying $iA^2_{+}(M)$
 with $i \mathfrak{su}(W_{+})$ via the isometry $\rho_{\xi}$.
\item Recalling that $\delta_-$ is an isometry from $\sym_0(TM, g_{\xi})$ to $\Hom(\Lambda^2_-T^*M, \Lambda^2_+T^*M)$ with the given norms, we immediately get that the adjoint of the map 
$s \rMapsto \delta_-(s_0)(F_A^-)$ is given by $\theta \rMapsto 2 (F_A^-)^* \tens \theta$. 
\item To compute the adjoint  of the map $s \rMapsto \rho_{\xi}(\di s)\psi = j_{\psi} \circ \di (s)$  recall that the adjoint of the divergence operator 
$\di : \sym(TM, g_{\xi}) \rTo A^1(M)$ is given by the map: 
$  \sigma \rMapsto - (1/2) \;L_{\sigma^{\sharp}}g_{\xi}$, where we 
indicate with $\sigma^{\sharp}$ the vector field obtained from the 1-form 
$\sigma$ by raising the indexes.
The adjoint of $\rho_{\xi}(\di (-))\psi$ is then: 
$ \chi \rMapsto - L_{\pr (\psi^* \tens \chi)}g_{\xi} $. 
\item With similar arguments one can 
prove that the adjoint of $s \rMapsto \rho_{\xi}(d \tr s)\psi$ is given by 
$ \chi \rMapsto  d^*(\pr(\psi^* \tens \chi)) g_{\xi}$. \item Denote with 
$\nabla^{W}_A \psi^*$ the linear map: $TM \rTo W_+^*$ defined by: $X \rMapsto 
\langle -, \nabla^{W,\xi}_{A,X} \psi \rangle$ and with $\pr(\nabla^{W}_{A} \psi^* \tens \chi)$ the 
$2$-tensor defined by: $(X,Y) \rMapsto \langle Y, \pr(\nabla^{W, \xi}_{A,X} \psi^* \tens \chi)\rangle$. The adjoint of the map: $\sym(TM, g_{\xi}) \rTo \Gamma(W_-)$, 
defined by $s \rMapsto \rho \circ s^* \circ \nabla^{W, \xi}_A \psi$, is the map: 
$ \chi \rMapsto \sym \pr(\nabla^{W}_A \psi^* \tens \chi)$. This can be proved 
expressing everything in a local orthonormal frame $e^i$ and recalling the identity~(\ref{adj}). 
\end{enumerate}
We are ready to write down the formal adjoint of the operator 
$D_{(A, \psi, g_{\xi})}\tilde{\mbb{F}}$:
\begin{pps}The formal adjoint 
 of the operator  $D_{(A, \psi, g_{\xi})}\tilde{\mbb{F}}$ is the differential operator: $$ (D_{(A, \psi, g_{\xi})}\tilde{\mbb{F}})^*: \Gamma(W_- ) 
\oplus A^2_+(M, i\mbb{R}) \rTo A^1(M,i\mbb{R}) \oplus \Gamma(W_+) \oplus \sym(TM,g_{\xi})$$ given by: \small
$$ (D_{(A, \psi, g_{\xi})}\tilde{\mbb{F}})^*(\chi, \theta) = \matrice{c}{
d^*\theta +i \: \pim (\psi^* \tens \chi)  \\ 
 \\
D_A^{\xi} \chi -\dfrac{1}{2} \rho_{\xi}(\theta)\psi \\
 \\
-\sym \pr (\nabla^{W}_A \psi^* \tens \chi) +\dfrac{1}{2} L_{\pr (\psi^* \tens \chi)} g_{\xi}
+\dfrac{1}{2} d^* \pr (\psi^* \tens \chi)g_{\xi} -\dfrac{1}{2} (F_A^+, \theta) g_{\xi} -2 (F_A^-)^* \tens \theta  }
$$\normalsize
\end{pps}
\subsection{The obstruction to transversality}By (\ref{BBB}) we know that $\ker (D_{(A, \psi, \xi)}\mbb{F}_{\Met(M)})^* = \ker (D_{(A, \psi, g_{\xi})}\tilde{\mbb{F}})^*$; hence
the equations for the kernel of $(D_{(A, \psi, \xi)}
\mbb{F}_{\Met(M)})^*$ read: 
\begin{subequations}\label{ker}\begin{gather}\label{ker1}
d^* \theta +i \: \pim (\psi^* \tens \chi) =0 \\ \label{ker2}
D_A^{\xi} \chi -\frac{1}{2} \rho_{\xi}(\theta)\psi =0 \\ \label{ker3}
-\sym  \pr (\nabla^{W}_A \psi^* \tens \chi) +\frac{1}{2} L_{\pr (\psi^* \tens \chi)}g_{\xi} 
+\frac{1}{2} d^* \pr (\psi^* \tens \chi)g_{\xi} -\frac{1}{2} (F_A^+, \theta) g_{\xi} - 2(F_A^-)^* \tens \theta =0   
\end{gather}
\end{subequations}where $(A, \psi, \xi)$ satisfies $\mbb{F}_{\Met(M)}(A, \psi, \xi)=0$, with $\psi \neq 0$.
Equations (\ref{ker}) can be slightly simplified.
\begin{lemma}\label{icci} If $(\chi, \theta)$ is a solution of equations (\ref{ker}), then 
$ ( F_A^+, \theta)=0$ and $ \di (\psi^* \tens \chi) =0$. 
\end{lemma}\begin{proof}Consider the equations (\ref{ker}). 
Applying the operator $d^*$ to the first equation we get
$ d^* \pim (\psi^* \tens \chi) = -\di \pim (\psi^* \tens \chi)=0$; hence
$\di  (\psi^* \tens \chi) = \di  \pr (\psi^* \tens \chi)$. 
Recall  the following identity\footnote{One can easily prove the equality estabishing it first at the symbol level; then, showing pointwisely 
the equality at the zero-th order terms taking an adapted 
orthonormal frame, that is a local orthonormal frame $e^i$ such that 
$\nabla e^i (p)=0$ at the point $p$.}: if $\phi$ is a positive spinor, and $\zeta$ is a negative one, then 
\begin{equation}\label{diracdiv} 2 \, \di (\phi^* \tens \zeta) = \langle D_A  \phi, \zeta \rangle -
\langle \phi, D_A \zeta  \rangle \;.
\end{equation}
We now take the trace in the third equation, remembering that, for 
any vector field $X$, we have $\tr L_X g_{\xi} = 2
 \di X$. We get: 
$$ -\tr \sym \pr (\nabla^{W}_A \psi^* \tens \chi) +\frac{1}{2}\tr
L_{\pr (\psi^* \tens \chi)} g_{\xi} +2 \: d^* \pr (\psi^* \tens \chi) - 2(F_A^+, \theta) = 0 \;,$$or, equivalently, since $\di (\psi^* \tens \chi)$ is real,
\begin{equation}
\label{lollo}
     \di (\psi^* \tens \chi) + 2(F_A^+, \theta) =
     0 \;,\end{equation}since $\tr L_{\pr (\psi^* \tens \chi)} g_{\xi} =
   2 \di \pr (\psi^* \tens \chi)$, and, by a simple 
computation taking a orthonormal frame, 
$\tr \sym \pr (\nabla^{W}_A \psi^* \tens \chi) = 1/2 \, \pr \langle D_A \psi, \chi \rangle = 0$. 
Now, taking the scalar product with $\psi$ in the second equation we get: 
$ \langle \psi, D_A \chi \rangle -1/2 \:\langle \psi, \rho_{\xi}(\theta) \psi 
\rangle =0$, which becomes, using (\ref{ide}) and (\ref{diracdiv}): 
\begin{equation}\label{lollona}
   \di  (\psi^* \tens \chi) + ( F_A^+, \theta  
) =0 \;.\end{equation}Combining (\ref{lollo}) and (\ref{lollona}) we get the result. \end{proof}
\begin{pps}
The obstruction to the transversality of the universal Seiberg-Witten 
functional $(\mbb{F}_{\Met(M)})^{2,r}_p$ at the solution\footnote{Here we consider 
$A$, $\psi$, $\xi$ of class $C^{\infty}$} $(A, \psi, \xi)$ of the universal Seiberg-Witten equations is given by nontrivial solutions $(\theta, \chi) \in i A^2_+(M) \oplus \Gamma(W_-)$ to the following equations:
\begin{subequations} \label{ke}
\begin{gather} \label{ke1}
d^* \theta +i\: \pim (\psi^* \tens \chi) =0 \\ \label{ke2}
D_A^{\xi} \chi -\frac{1}{2} \rho_{\xi}(\theta)\psi =0 \\ \label{ke3}
-\sym \: \pr (\nabla^W_A \psi^* \tens \chi) 
+\frac{1}{2} L_{\pr (\psi^* \tens \chi)} g_{\xi} -2(F_A^-)^* \tens \theta =0  \\ \label{ke4}
(\theta, F_A^+)=0 \\ \label{ke5}
\di (\psi^* \tens \chi) = 0 
\end{gather}\end{subequations}
\end{pps}
\begin{proof}By (\ref{ecci}) the cokernel of the differential 
$D_{(A, \psi, \xi)} (\mbb{F}_{\Met(M)})^{2,r}_p$ coincides with the kernel of the formal adjoint of the 
differential $D_{(A, \psi, \xi)}\mbb{F}_{\Met(M)}$  on 
sections of class $C^{\infty}$. The equations (\ref{ker}) of the kernel of $(D_{(A, \psi, \xi)}\mbb{F}_{\Met(M)})^*$ are now equivalent, by lemma \ref{icci}, to equations~(\ref{ke}). 
\end{proof}
\begin{remark}\label{EF}
We discuss now the gaps in the proof of the transversality with generic metrics by Eichhorn and Friedrich. The two authors (in
\cite[Proposition 6.4]{EichhornFriedrich1997} and Friderich alone in
\cite[page 141]{FriedrichDORG}) try to prove directly that the differential 
$D_{(A, \psi, g_{\xi})} \tilde{\mbb{F}}$ of the
perturbed Seiberg-Witten functional is surjective. A first source of 
unclearness is that they never give a precise expression of the
variation of the Dirac operator, which we have seen as being a
fundamental difficulty in the question; in particular no mention is
made about the term $-\rho_{\xi} \circ s^* \circ \nabla_A^{W} \psi$.
The authors take into account variations of the metric which are
orthogonal to the orbits of the action 
the diffeomorphism group ${\rm Diff}(M)$ on $\Met(M)$: this condition is
precisely expressed by
$\di s =0$. They now remark that the variation of the second
equation involves just the traceless part of the tensor $s_0$: as a
consequence, they now claim that they can deal with conformal perturbations
separately from volume preserving ones. Thanks to this uncorrect
argument, as we will see, they get to the two separate conditions, reading, 
our notations:
 \begin{gather*}
\langle \dd{}{g}(*_{g})(s_0) F_A , \theta \rangle=0 \;, \quad
 \langle \rho ( df) \psi, \chi \rangle =0 
\end{gather*}which are to be satisfied by an element $(\chi, \theta)$
in the cokernel of $D_{(A, \psi,g)} \tilde{\mbb{F}}$, for all 
 $s_0 \in \sym_0(TM,g)$ and for all   $f \in
C^{\infty}(M, \mbb{R})$ such that $ \di s_0 = df$. The result would
follow from them  (remark that they correspond\footnote{without
  taking into account the condition $\di s_0 =df$}, taking formal adjoints, 
 to two separate equations: $(F_A^{-})^* \tens \theta =0$; $d^* \pr
 (\psi^* \tens \chi)=0$).
This argument is not correct for the following two reasons. Firstly,
the term $-\rho \circ s^* \circ \nabla ^W_A \psi$ depends on the full
tensor $s$ and not just on his trace; secondly the variation of the
second equation does not involve just volume preserving perturbations,
since conformal perturbations come to play a role in the
identification $i \Lambda^2_+ T^*M \simeq i \mf{su}(W_+)$ via
$\rho_{\xi}$.
\end{remark}
\section{Transversality over K\"ahler monopoles}
\subsection{K\"ahler monopoles}We now consider the transversality problem 
on K\"ahler surfaces. Let $(M, J)$ be a compact connected $4$-manifold 
with an integrable complex structure $J$. We will indicate with $H_J(M)$
the space of hermitian metrics  with respect to the complex structure 
$J$; it is 
a splitting Fr\'echet 
submanifold of the manifold of riemannian metrics $\Met(M)$.
If $g \in  H_J(M)$, we indicate with $\omega_g:=g(-,J(-))$ the $(1,1)$ form associated to~$g$.
Suppose now that $(M,J)$ is of K\"ahler type. Let $K_J(M)$ be the set of K\"ahler
metrics for the complex structure $J$: 
$ K_J(M) := \{g \in H_J(M) \; | \;  \; d \omega_g=0 \}$.
A K\"ahler surface is by definition a $4$-manifold  
with  a $U(2)$-reduction of the structural group of the tangent bundle $P_{U(2)} \rInto P_{GL(4)}$ admitting a
torsion free $U(2)$-connection. The natural morphism: 
$i\colon U(2) \rInto SO(4) \times U(1)$ lifts to a morphism 
$j \colon U(2) \rInto Spin^c(4)$ so that $\nu \circ j =i$. The \emph{canonical 
$Spin^c$-structure} $\xi_0$ on a K\"ahler manifold $M$ is then given by the 
$\mu$-equivariant map: 
$$ \xi_0: Q_{Spin^c(4)}:= P_{U(2)} \times_{j} Spin^c(4) \rTo P_{GL_+(4)} $$induced 
by the morphism $j$. Remark that $Q_{SO(4)} \simeq P_{U(2)} \times_{U(2)} SO(4)$
and that $Q_{U(1)} \simeq P_{U(2)} \times_{\det} U(1)$. As a consequence the spinor 
bundle is: $W:= \Lambda^{0,*}T^*M$, with 
$W_{+} \simeq \Lambda^{0, \textrm{even}}T^*M $, $W_- \simeq \Lambda^{0, 1}T^*M$. The fundamental line bundle $L$ is isomorphic to the anticanonical bundle 
$\det W_+ \simeq K^*_M$ and the fundamental class $c$ is $c_1(M)$. 
The Clifford multiplication of the structure $\xi_0$ is given by 
\begin{diagram}[height=0.5cm]
\rho_{\xi_0} \colon T^*M & \rTo & \End(\Lambda^{0,\textrm{even}}T^*M,
\Lambda^{0,1}T^*M  ) \\ 
\qquad x & \rMapsto & \sqrt{2}[x^{0,1} \wedge (\cdot) - x^{0,1}\lrcorner (\cdot)]
\end{diagram}
Any other $Spin^c$-structure $\xi_N$ is obtained, up to isomorphism, from 
the canonical one by twisting the spinor bundle by a line bundle $N \in \Pic_{\rm top}(M)$; the resulting bundle of spinors is $W= \Lambda^{0,*}T^*M \tens N$, the determinant line 
bundle is twisted by of $N^{\tens 2}$, $L = 
K_M^* \tens N^{\tens 2}$, and the fundamental class changes as $c= c_1(M) + 2 c_1(N)$. The Clifford multiplication for $\xi_N$ is 
$\rho_{{\xi}_0} \tens \id_N$. 
In the notations of section \ref{2.5}, let $\Xi_{H_J(M)}$ 
be the \emph{hermitian} $Spin^c$-structures of class $c$ and fixed
type, that is, the 
$Spin^c$-structures of class $c$ in $\Xi$
projecting onto $J$-hermitian metrics, and let $\Xi_{K_J(M)}$ be
the \emph{k\"ahlerian} ones (those 
projecting onto  K\"ahler metrics). We will call the parametrized Seiberg-Witten  moduli space $\mc{M}_{H_J(M)}$ and $\mc{M}_{K_J(M)}$ the moduli spaces of \emph{hermitian} and \emph{k\"ahlerian monopoles}, respectively. 

To express Seiberg-Witten equations on a K\"ahler surface $(M,g,J)$ for the $Spin^c$-structure $\xi_N$, we fix the 
Chern connection $A_{K_M}$ on $K_M$ and make the changement of variables: 
$\mc{A}_{U(1)}(N)  \simeq \mc{A}_{U(1)}(L)$ given by
$A  \rMapsto  A^*_{K_M} \tens A^{\tens 2}$. 
The Dirac operator for this 
$Spin^c$-structure and for $A \in \mc{A}_{U(1)}(N)$ is: 
$D_A := \sqrt{2}(\dbar_A + \dbar_A^*) $. The Seiberg-Witten equations on a 
compact K\"ahler surface for a spinor $(\alpha, \beta) \in A^{0,0}(N) \oplus 
A^{0,2}(N)$ and for a $U(1)$-connection $A$ on $N$ read:
\begin{gather*}
\dbar_A \alpha + \dbar_A^* \beta =0 \\ 
F^{0,2}_A = \frac{\bar{\alpha}\beta}{2} \\
2  F_A^{1,1} -  F_{K_M} = i \frac{|\alpha|^2-|\beta|^2}{4} \omega_g
\end{gather*}
where we  split the second equation according the splitting of 
self dual $2$-forms in $ \Lambda^2_+T^*M \tens \mbb{C} \simeq 
 \Lambda^{2,0}T^*M \oplus  \Lambda^{0,2}T^*M \oplus \mbb{C} \omega_g $. 
It is well known that if $\deg(L)< 0$ then $(A, \alpha, \beta)$ is a solution 
of the Seiberg-Witten equations if and only if $\dbar_A$ is a 
holomorphic structure for $N$, $\alpha$ is a non zero holomorphic section of 
$(N, \dbar_A)$ and $\beta=0$. Analogously, if $\deg(L)> 0$, we have 
a solution whenever $\alpha =0$, $\dbar_{A^* \tens A_{K_M}}$ is a holomorphic 
structure of $N^* \tens K_M$ and $\sharp \beta$ is a non 
zero holomorphic section of 
$(N^* \tens K_M, \dbar_{A^* \tens A_{K_M}})$, where $\sharp$ denotes here the 
complex Hodge star operator; the involution $\jmath: (A, \alpha, \beta) 
\rMapsto (A^* \tens A_{K_M}, \sharp \beta, \sharp \alpha)$ 
exchanges solutions of
Seiberg-Witten equations for the $Spin^c$ structure $\xi_N$ and
solutions for the $Spin^c$-structure $\xi_{N^* \tens K_M}$.
Moreover $(A, \alpha, \beta)$ is a reducible solution if and only if $\deg(L)=0$ and $A$ is self-dual. As a consequence 
$K_J(M) \cap \Met(M)_{c-\textrm{good}} =\{ g \in K_J(M) \:| \; [\omega_g] \cup c 
\neq 0 \}$. Therefore: $$ \M^{*}_{K_J(M)}=\mc{M}_{K_J(M)} \cap \mc{M}_{\Met(M)}^* = \M_{K_J(M)} \cap \M^{**}_{\Met(M)} 
= \M^{**}_{K_J(M)} \;.$$
In this 
section we will prove the following theorem: 
\begin{theorem} \label{teorema}
The parametrized moduli space of hermitian Seiberg-Witten monopoles $(\M_{H_J(M)})^{2,r}_p$, and hence the universal moduli space  $(\mc{M}_{\Met(M)})^{2,r}_p$, is smooth at 
irreducible k\"ahlerian monopoles~$\M_{K_J(M)}^*$.
\end{theorem}
By remarks \ref{cinfty}, \ref{postequi}, \ref{postequi2}, we can paraphrase this theorem as:
\begin{theorem}
Let $(M,g,J)$ a K{\"a}hler surface. Let $N$ a hermitian line
  bundle on $M$ such that $2 \deg(N) -\deg(K_M) \neq 0$.  Consider the $Spin^c$-structure $\xi_N$, obtained by twisting the canonical one
 with the hermitian line bundle $N$. 
For a generic hermitian metric $h \in H_J(M)$ in a small open neighbourhood of $g$ and for any $Spin^c$-structure $\xi^{\prime}$ of fundamental class 
$c(\xi_N)= c_1(M) + 2 c_1(N)$, compatible with $h$,  
 the Seiberg-Witten moduli space $\mc{M}_{\xi^{\prime}}^{SW}$
  is smooth of the expected dimension. 
The statement holds as well for a generic riemannian
metric $h \in \Met(M)$ in a small open neighbourhood of $g$.
\end{theorem} 

In the next subsection we make use of the complex structure $J$ to split the
symmetric endomorphisms $\sym(TM,g)$ of $TM$ with respect to a 
 $J$-hermitian metric $g$ in hermitian and anti-hermitian ones. In subsection \ref{obsKahl} we write down equations (\ref{ke}) in the K\"ahler context and in subsection \ref{proofthm} we will finally prove theorem \ref{teorema}.
 
\subsection{A decomposition for symmetric $2$-tensors}\label{subsect:decomp}
The endomorphisms $\End(TM)$ of the tangent bundle $TM$ decompose, thanks to the complex structure~$J$, in $J$-linear and $J$-antilinear ones:
$\End(TM) \simeq \End(TM,J) \oplus \overline{\End(TM,J)}$. 
Consider now a metric $g$, \emph{hermitian} with respect to $J$.
The previous decomposition of $\End(TM)$ induces a decomposition of
the symmetric 
endomorphism $\sym(TM,g)$ of $TM$, with respect to $g$, in hermitian and 
anti-hermitian ones: 
\begin{equation}\label{decosym}
\sym(TM,g) \simeq \mathfrak{u}(TM,J) \oplus \mathfrak{su}(TM,J) \;,
\end{equation}where $\mathfrak{u}(TM,J) = \sym(TM,g) \cap \End(TM,J)$ and $\mathfrak{su}(TM,J) = \sym(TM,g) \cap \overline{\End(TM,J)}$. 
Analogously, symmetric $2$-tensors in $S^2 T^*M$ can be decomposed in the direct sum 
$S^2(T^*M) \simeq S^{1,1}T^*M \oplus S^2_{AH}T^*M$
of hermitian $2$-tensors 
$S^{1,1}T^*M = 
\{ s \in S^2T^*M \: | \; s(JX, JY) =s(X,Y) \; \; \forall \;  X,Y \in TM 
\}$ and antihermitian ones: $S^2_{AH}T^*M = \{ s \in S^2T^*M \: | \; s(JX, JY) = -s(X,Y) \; \; \forall \: X,Y \in TM 
\}  $. 
The decompositions for $\sym(TM,g)$ and $S^2T^*M$ identify one to the other 
once we identify  tangent and  cotangent bundle by means of the metric $g$.

Consider now the complexified tangent bundle $TM \tens \mbb{C}$. The
complex 
symmetric $2$-tensors $S^2(T^*M \tens \mbb{C})$ split, according to the decomposition
$TM \tens \mbb{C} =T^{1,0}M \oplus T^{0,1}M$ as $$
S^2(T^*M \tens \mbb{C}) = S^{2,0}T^*M \oplus S^{0,2}T^*M \oplus
S^{1,1}_{\mbb{C}} T^*M $$where we indicate $S^2(\Lambda^{1,0}T^*M)$
with $S^{2,0}T^*M$,  
$S^2(\Lambda^{0,1}T^*M)$ with $S^{0,2}T^*M$ and with
$S^{1,1}_{\mbb{C}}T^*M$ the subbundle of $\Lambda^{1,0}T^*M \tens
\Lambda^{0,1}T^*M \oplus \Lambda^{0,1}T^*M \tens \Lambda^{1,0}T^*M$
invariant by the transposition of factors $\tau$ in the tensor
product; in these notations the hermitian $2$-tensors $S^{1,1}T^*M$
introduced above coincide with the subspace of real tensors in
$S^{1,1}_{\mbb{C}}T^*M$, that is, tensors invariant by conjugation.
 Let now $s \in S^2T^*M$, extended by $\mbb{C}$-linearity to 
the element  $s_{\mbb{C}}\in S^2(T^*M \tens \mbb{C})$; according to the above decomposition, $s_{\mbb{C}}$ can be written as
$s_{\mbb{C}}= s_{2,0} + s_{0,2} + s_{1,1}$, with $s_{2,0}=\overline{s_{0,2}}$ and $\overline{s_{1,1}}=s_{1,1}$. It is clear that $s \in S^{1,1}T^*M$ if and only if $s_{0,2}=0$;
in this case $s_{1,1}$ defines an hermitian form 
on $T^{1,0}M$; hence $S^{1,1}T^*M \simeq \Herm(T^{1,0}M)$. On the other hand, 
$s \in S^2_{AH}T^*M$ if and only if $s_{1,1}=0$; in this case $s_{2,0}$ and $s_{0,2}$ define quadratic forms on $T^{1,0}M$ and $T^{0,1}M$, respectively, one conjugated of the other. Hence $S^2_{AH}T^*M \simeq S^{2,0}T^*M \simeq S^{0,2}T^*M$.

Using the complexified metric 
$g_{\mbb{C}}$ to identify $T^*M \tens \mbb{C}$ with 
$TM \tens \mbb{C}$, the previous considerations can be stated 
for $\mbb{C}$-linear extensions of symmetric endomorphisms $f \in \sym(TM,g)$ 
to 
$f_{\mbb{C}} \in \End(TM \tens \mbb{C})$. An endomorphism $f \in \End(TM)$ extends by 
$\mbb{C}$-linearity to an endomorphism $f_{\mbb{C}} \in \End(TM \tens \mbb{C})$ such that 
 $f_{\mbb{C}}(\bar{z}) = \overline{f_{\mbb{C}}(z)}$ for all $z \in TM \tens \mbb{C}$. According to the decomposition $TM \tens \mbb{C} \simeq T^{1,0}M \oplus T^{0,1}M$ this extension can be written as: 
 \begin{equation}\label{decomorf} f = \matrice{cc}{a & \bar{b} \\ b & \bar{a}} \;. 
 \end{equation}The endomorphism $f$ is then $J$-linear if and only if $b=0$, $J$-antilinear if and only if $a =0$. Moreover, $f$ is symmetric with respect to $g$ if and only if  
$(Z,W) \rMapsto g(a(Z),\bar{W})$ is an hermitian form on $T^{1,0}M$ and
 $(Z, W) \rMapsto g(b(Z), W)$ is a complex quadratic form on $T^{1,0}M$. 
 Hence we can identify $\mf{u}(TM,J) \simeq \Herm(T^{1,0}M)$; $\mf{su}(TM,J) \simeq S^{2,0}T^*M$. Analogously, using $\bar{a}$ and $\bar{b}$ we get identifications $ \mf{u}(TM,J) \simeq \Herm(T^{0,1}M)$; $\mf{su}(TM,J) \simeq S^{0,2}T^*M$.
\begin{remark}\label{decormk3}The space of hermitian $2$-tensors $S^{1,1}T^*M$ is isomorphic to the space of real $(1,1)$ forms 
$\Lambda^{1,1}_{\mbb{R}}T^*M$ via the isomorphism: $s \rMapsto  
s(\cdot , J(\cdot))$. In local coordinates, if $a \in S^{1,1}T^*M \simeq 
\Herm(T^{1,0}M)$ is given by $a= \sum_{i,j} a_{i\bar{j}}dz_i \tens d\bar{z}_j$, with $a_{i\bar{j}}$ an hermitian matrix, the associated real $(1,1)$-form 
is given by $-2i \sum_{ij} a_{i \bar{j}} dz_i \wedge  d\bar{z}_j$.
\end{remark}
\begin{remark}\label{rmk:decouseful}Let $u \in \Lambda^{1,0}T^*M \tens \Lambda^{0,1}T^*M$. Let $\sigma(u)$ the $(1,1)$-form in $\Lambda^{1,1}T^*M$ obtained by the projection of $u$ on 
$\Lambda^2(T^*M \tens \mbb{C})$. 
Denote moreover with $\herm u$ the hermitian part of the sesquilinear form
on $T^{1,0}M$ defined by $u$: $\herm u = 1/2 (u + \overline{\tau(u)})$.
Then 
$\sym \pr  u \in \mf{u}(TM,J)$ and coincides with 
$1/2 \herm u = 1/4 \: (u + \overline{\tau(u)})$
as hermitian form on $T^{1,0}M$. By the previous remark, 
the associated real $(1,1)$-form to $\sym \pr u $ is $-i/2 \: (\sigma(u) - \overline{\sigma(u)})$.
Let now $v \in \Lambda^{0,1}T^*M \tens \Lambda^{0,1}T^*M$. Then $\sym \pr v \in 
\mf{su}(TM,J)$ and coincides with $1/2 \sym v \in S^{0,2}T^*M$, in the identification $\mf{su}(TM,J) \simeq S^{0,2}T^*M$.
\end{remark}
We need now to take into account the decomposition (\ref{decosym}) in the isometry: $ \delta_- \colon \sym_0(TM,g) \rTo \Hom(\Lambda^2_-T^*M, \Lambda^2_+T^*M)$. 
We identify
$\Lambda^2_-T^*M$ with $\Lambda^{1,1}_{\omega_g^{\perp}, \mbb{R}}$, that
is, with the real $(1,1)$-forms orthogonal to the Kahler form $\omega_{g}$,
and $\Lambda^2_+ T^*M$ with $\Lambda^{0,2}T^*M \oplus \mbb{R} 
\omega_{g}$.   
If $f \in \sym(TM,g)$, let $a(f) \in \End(T^{1,0}M)$ and $b(f) \in
\Hom(T^{1,0}M, T^{0,1}M)$ be the components of the extension of $f$ to
$TM \tens \mbb{C}$
seen in (\ref{decomorf}). Set $\mathfrak{u}_0(TM,J) = \mathfrak{u}(TM,J) \cap 
\sym_0 (TM,g)$. With this notations we have:
\begin{lemma}\label{lmm:decomptype}For all $f  \in \mathfrak{u}_0(TM,J)$ then 
$ \delta_{-}(f) \Lambda ^{1,1}_{\omega_g^{\perp}, \mbb{R}} \subseteq \mbb{R} 
\omega_{g}$. Therefore the isometry $\delta_{-} \colon \sym_0(TM,g) \rTo 
\Hom(\Lambda^2_-T^*M,\Lambda^2_+T^*M)$ splits, according to the  decomposition (\ref{decosym}),~as~:
\begin{diagram}[height=.5cm]
\mathfrak{u}_0(TM,J) \oplus \mathfrak{su}(TM,J) &
\rTo &
\Hom(\Lambda^{1,1}_{\omega_g^{\perp}, \mbb{R}}, \Lambda^{0,2}T^*M)
\oplus  \Hom(\Lambda^{1,1}_{\omega_g^{\perp}, \mbb{R}}, \mbb{R} \omega_{g}) \\  
\quad \quad  ( s ,t) \quad \quad & \rMapsto & \; \; (\; 
\delta_{-}(\bar{b}(t)) \;,  \; \delta_{-}(s) \; )
\end{diagram}    
\end{lemma}\begin{proof}Let $(s,t) \in \mathfrak{u}_0(TM,J) \oplus \mathfrak{su}(TM,J)$. The 
derivation $i(s^*)$
induced by an element $s \in 
\mathfrak{u}_0(TM,J)$ preserves the spaces $\Lambda^{1,1}T^*M$, 
$\Lambda^{2,0}T^*M$ and $\Lambda^{0,2}T^*M$, because $s$ is
$J$-linear. Therefore, for such $s$, $\delta_{-}(s)  
\Lambda ^{1,1}_{\omega_g^{\perp}, \mbb{R}} \subseteq \Lambda^{1,1}T^*M$, but
by definition $\delta_{-}(s) \Lambda^2_-T^*M \subseteq
\Lambda^2_+T^*M$; as a result $\delta_{-}(s) \Lambda^{1,1}_{\omega_g^{\perp}, 
\mbb{R}} \subseteq 
\mbb{R}\omega_{g}
$. Any $t \in \mathfrak{su}(TM,J)$
 is $J$-antilinear, hence its extension to $TM \tens
\mbb{C}$ exchanges $T^{1,0}M$ and $T^{0,1}M$; consequently $i(t^*)
\Lambda^{1,1}T^*M \subseteq \Lambda^{0,2}T^*M \oplus \Lambda ^{2,0}T^*M$. We
can write $t^* = b(t)^* + \bar{b}(t)^*$, with $b(t)^* : \Lambda^{0,1}T^*M
\rTo \Lambda^{1,0}T^*M$, and $\bar{b}(t)^* :  \Lambda^{1,0}T^*M
\rTo \Lambda^{0,1}T^*M$. Therefore $i(b(t)^* )\Lambda^{1,1}T^*M
\subseteq \Lambda ^{2,0}T^*M$ and $i(\bar{b}(t)^*)\Lambda^{1,1} T^*M
\subseteq \Lambda^{0,2}T^*M$. Therefore in the splitting $$  
\Hom(\Lambda^2_-T^*M,\Lambda^2_+T^*M) \simeq
\Hom(\Lambda^{1,1}_{\omega_g^{\perp}, \mbb{R}}, \Lambda^{0,2}T^*M)
\oplus  \Hom(\Lambda^{1,1}_{\omega_g^{\perp}, \mbb{R}}, \mbb{R} \omega_g)
$$ the element $(s,t)$ acts as $\delta_{-}(\bar{b}(t)) \oplus
\delta_{-}(s)$. \end{proof}
\subsection{The obstruction to the transversality on a K\"ahler monopole}\label{obsKahl}
As discussed in subsection \ref{sbvc}, in order to prove theorem \ref{teorema} we need to prove that the intrinsic differential $D_x \Psi_{H_J(M)}$ of the section 
$\Psi_{H_J(M)} : (\mathfrak{B}^*_{H_J(M)})^{2,r}_p \rTo \mf{E}_{H_J(M)}$ 
is surjective at an irreducible k\"ahlerian monopole $x= ([A, \psi],
\xi)$. 
We can suppose that $\xi= \xi_N$ for a certain $N \in \Pic_{\rm
  top}(M)$; moreover,
because of the involution $\jmath$, it is not at all 
restrictive to take $x$ a k\"ahlerian monopole with negative degree.
By subsection \ref{obstransv} and in particular remarks 
\ref{transvrmk1}, \ref{transvrmk2},  the obstruction to the surjectivity of 
$D_x \Psi_{H_J(M)}$ is given by the kernel of the operator 
\begin{equation}\label{ope}
(D_{(A, \psi, g_{\xi})} \tilde{\mbb{F}}_{H_J(M)})^* \simeq 
(\id_{T_{(A, \psi)}\mc{C}} \oplus P_{T_{g_{\xi}} H_J(M)}) \circ (D_{(A, \psi, g_{\xi})} \tilde{\mbb{F}})^* \end{equation}where $P_{T_{g_{\xi}} H_J(M)}$ is the orthogonal projection 
$T_{g_{\xi}} \Met(M) \rTo T_{g_{\xi}} H_J(M)$ onto the tangent space of hermitian metrics. Since, given the K\"ahler metric $g_{\xi}$, we can 
parametrize hermitian metrics with symmetric positive hermitian automorphisms 
$U^+(TM,J) = \Sym^+(TM,J) \cap \End(TM,J)$ with respect to the metric $g_{\xi}$, the tangent space to hermitian metrics is given by:$$ T_{g_{\xi}} H_J(M) \simeq T_{\id}U^+(TM,J) \simeq \mf{u}(TM,J) \;.$$
The form of the operator (\ref{ope}) implies that, to find the obstruction we want, we have to consider equations (\ref{ke}), with equation (\ref{ke3}) projected onto the component in $\mf{u}(TM,J)$, according to the decomposition~(\ref{decosym}).

We are now going to write down the 
kernel equations (\ref{ke}) on the K\"ahler monopole of negative degree 
$x= ([A, \psi], \xi_N)$, where $\psi = (\alpha, 0) \in A^{0,0}(N) \oplus A^{0,2}(N) $. For brevity's sake in the sequel we will indicate the K\"ahler metric $g_{\xi_N}$ just with $g$.
In what follows we will make the following identifications: \begin{itemize}
\item[a)] we will identify imaginary 1-forms in $iA^1(M)$ with $(0,1)$-forms in $A^{0,1}(M)$ via the isomorphism $A^{0,1}(M) \simeq iA^1(M)$ sending $\sigma \rMapsto \sigma -\bar{\sigma}$; 
\item[b)] the imaginary selfdual $2$-forms $iA^2_+(M)$ will be
  identified with forms in $i\mbb{R} \omega_g \oplus A^{0,2}(M)$, since we can always write 
$\theta \in iA^2_+(M)$ as $\theta = \lambda \omega_{g} + \mu -
\bar{\mu}$ for $\lambda \in i\mbb{R}$, $\mu \in A^{0,2}(M)$. 
We can therefore express the isomorphism $iA^2_+(M) \simeq i\mathfrak{su}(W_+)$ as (cf. \cite{MorganSWEAFMT}): 
$$ \lambda \omega_{g} + \mu \rMapsto 2 \matrice{cc}{\lambda & \mu \lrcorner(-) \\
\mu \w  (-) & - \lambda} \;,$$
where the matrix is written according to the decomposition $W_+ \simeq 
\Lambda^{0,0}T^*M \tens N \oplus \Lambda^{0,2}T^*M \tens N$.
\end{itemize}
 \begin{remark}\label{rmk: 1form}Remark that the 1-form $\phi^* \tens
   \zeta$, where 
$\phi \in A^{0,0}(N)$ and $\zeta \in   A^{0,1}(N)$, is given by~$ 1/\sqrt{2} \: \bar{\phi}\zeta \in A^{0,1}(M)$. 
\end{remark}
Equations (\ref{ke1}), (\ref{ke2}), (\ref{ke4}), (\ref{ke5}) now become easily, in the above identifications: 
\begin{subequations}
\begin{gather} \dbar^* \mu + \partial^*(\lambda \omega_{g}) +  \frac{1}{2\sqrt{2}} \bar{\alpha}\chi =0 \\ 
\sqrt{2} \dbar_A^* \chi - \lambda \alpha =0 \\ 
\sqrt{2} \dbar_A \chi - \mu \alpha =0 \\  
\lambda |\alpha|^2 =0 \\ 
\dbar^*(\bar{\alpha}\chi)=0 
\end{gather}
\end{subequations}
Since $\alpha$ is a nonzero holomorphic section of $(N, \dbar_A)$, we get from the fourth equation that $\lambda =0$ on the dense open set $M \setminus Z(\alpha)$ and hence everywhere. The last equation is, consequently, dependent from the first. 
Hence the equations are equivalent to:
 \begin{subequations}\begin{gather}  2\sqrt{2} \dbar^* \mu  +  \bar{\alpha}\chi =0 \\ 
\dbar_A^* \chi =0 \\ 
\sqrt{2} \dbar_A \chi - \mu \alpha =0 \\   
\lambda =0 
\end{gather}\end{subequations}
\paragraph{Contribution of the metric: hermitian perturbations.} We have now to write equation (\ref{ke3}) on the K\"ahler monopole $([A, (\alpha, 0)],\xi_N)$, according to the identifications made. We will use the decomposition (\ref{decosym}) and the identifications $\mf{u}(TM,J) \simeq \Herm(T^{0,1}M)$ and $\mf{su}(TM,J) \simeq 
S^{0,2}T^*M$ provided by subsection \ref{subsect:decomp}. Moreover we will 
identify $TM \tens \mbb{C}$ with $T^*M \tens \mbb{C}$
by means of the complexified metric $g_{\mbb{C}}$ (the
$\mbb{C}$-linear 
extension of $g$ to  to $T^*M \tens \mbb{C}$); it identifies $T^{1,0}M$ with $\Lambda^{0,1}T^*M$ and $T^{0,1}M$ with~$\Lambda^{1,0}T^*M$.

\paragraph{\it The term $\sym \pr (\nabla ^W_A \psi^* \tens \chi)$.} The linear map
$\nabla^W_A \psi^* \in W^*_+ \tens (T^*M \tens \mbb{C})$ becomes the form $\overline{\partial_A \alpha} : N^* \tens \Lambda^{0,1}T^*M$. The complex $2$-tensor 
$\nabla^W_A \psi^* \tens \chi$ 
can be identified, by remark \ref{rmk: 1form}, with the tensor 
$1/\sqrt{2} \: \overline{\partial_A \alpha} \chi \in 
\Lambda^{0,1} T^*M \tens \Lambda^{0,1}T^*M$. 
By remark \ref{rmk:decouseful} we can identify:
$$ \sym \pr (\nabla^W_A \psi^* \tens \chi) = \frac{1}{2\sqrt{2}}
\sym(\overline{\partial_A \alpha }  \chi) \in S^{0,2}T^*M \simeq 
\mathfrak{su}(TM,J) \;.$$
\paragraph{\it The term $L_{\pr  (\psi^* \tens \chi)}g$.} Recall that,
for a real 
vector field $X$, 
$ L_{X}g = 2 \sym \nabla^{g} X^{\flat} $, 
where $X^{\flat}$ denotes the 1-form obtained by $X$ lowering the 
indexes. As a consequence: $L_{\pr (\psi^* \tens \chi)}g = 2 \sym \nabla^{g}\pr ( \psi^* \tens \chi) = 2 \sym \pr \nabla^{g} ( \psi^* \tens \chi)$. By remark \ref{rmk: 1form}, the $2$-tensor $\nabla^{g} (\psi^* \tens \chi)$ is: 
$ \nabla^{g} (\psi^* \tens \chi) = 1/\sqrt{2} \:
 \nabla^{g} (\bar{\alpha} \chi) = 1/ \sqrt{2}\: [ D(\bar{\alpha} \chi) + \bar{D}(\bar{\alpha} \chi) ]$, where we denoted with $D$ and $\bar{D}$ the components $(1,0)$ and $(0,1)$ of $\nabla^{g}$, respectively. The term $\bar{D}(\bar{\alpha} \chi)$ is in 
$\Lambda^{0,1}T^*M \tens \Lambda^{0,1}T^*M $; 
the term $D(\bar{\alpha}\chi)$ is in 
$\Lambda^{1,0}T^*M \tens \Lambda^{0,1}T^*M$. 
Hence, by remark \ref{rmk:decouseful}
$$ L_{\pr (\psi^* \tens \chi)}g = 2 \sym \pr \nabla^{g} (\psi^* \tens \chi) = \frac{1}{\sqrt{2}} [ \herm D(\bar{\alpha}\chi) +
\sym \bar{D}(\bar{\alpha}\chi) ] \;,$$according to the decomposition
$S^2T^*M \simeq \Herm(T^{1,0}M) \oplus S^{0,2}T^*M \simeq
\mathfrak{u}(TM,J) \oplus \mathfrak{su}(TM,J)$. 
\paragraph{\it The term $(F_A^-)^* \tens \theta$.}Writing $\theta = \lambda \omega_{g} + \mu$, the term $(F_A^-)^* \tens \theta$ decomposes in the sum of 
$(F_A^-)^* \tens \lambda \omega_{g}$ and $(F_A^-)^* \tens \mu$. 
The map $\theta \rMapsto 2 (F_A^-)^* \tens \theta$ was built as the adjoint of the map $s \rMapsto \delta_-(s_0) F_A^-$. 
As a consequence of lemma \ref{lmm:decomptype}, $(F_A^-)^* \tens \lambda \omega_{g}$ is in $\mathfrak{u}_0(TM,J) \simeq S^{1,1,}_0T^*M$ (the traceless tensors in $S^{1,1}T^*M$), 
while $(F_A^-)^* \tens \mu$ is in $\mathfrak{su}(TM,J) \simeq S^{0,2}T^*M$.
\paragraph{\it Contribution of hermitian perturbations.} Equation (\ref{ke3}) splits in the two following equations, according to the decomposition (\ref{decosym}): 
\begin{subequations}
\begin{gather} 
\herm D(\bar{\alpha}\chi) -4\sqrt{2} (F_A^-)^* \tens \lambda \omega_{g} =0  \\
- \sym(\overline{\partial_A \alpha}  \chi ) + \sym D(\bar{\alpha}\chi) -4\sqrt{2}
(F_A^-)^* \tens \mu=0 
\end{gather}
\end{subequations}
Identifying elements in $\mathfrak{u}(TM,J)$ with real $(1,1)$-forms, as seen in remarks \ref{decormk3}, \ref{rmk:decouseful}, the first equation become: 
$$
-i( \partial(\bar{\alpha}\chi) - \dbar(\alpha\bar{\chi})) -4\sqrt{2} (F_A^-)^* \tens \lambda \omega_{g} =0 \;.$$It represents the contribution to transversality coming from 
hermitian perturbations of the K\"ahler metric~$g$.
\subsection{Proof of  theorem \ref{teorema}}\label{proofthm}
We  have to prove the surjectivity of the 
differential $D_{(A, \psi, g)} \tilde{\mbb{F}}_{H_J(M)}$ 
on a K\"ahlerian monopole $([A, \psi], \xi) = ([A, (\alpha, 0)], \xi_N)$. The obstruction to the transversality is 
given by a nontrivial solution to the equations of 
$\ker (D_{(A, \psi, g)} \tilde{\mbb{F}}_{H_J(M)})^*$: 
\begin{subequations}
\begin{gather} 2\sqrt{2} \dbar^* \mu +   \bar{\alpha}\chi =0 \\ 
\dbar_A^* \chi =0 \\ 
\sqrt{2} \dbar_A \chi - \mu \alpha =0 \\  
\lambda =0 \\
\partial(\bar{\alpha}\chi) - \dbar(\alpha\bar{\chi})=0
\end{gather}
\end{subequations} 
This system of partial differential equations does not have any nontrivial solution: in order to see this, we apply the operator $\partial^*$ to the last equation, obtaining: 
$$ \Delta_{\partial}(\bar{\alpha}\chi) - \partial^* \dbar(\alpha \bar{\chi}) =0 \;.$$
Using the K\"ahler identity $\partial^* \dbar + \dbar \partial^*=0$, we get 
that $\partial^* \dbar(\alpha \bar{\chi}) = \dbar \partial^* (\alpha \bar{\chi}) =0$, since we already know that $\dbar^*(\bar{\alpha}\chi)=0$. We are left with 
$ \Delta_{\partial}(\bar{\alpha}\chi)=0$, that is, $\bar{\alpha}\chi$ is 
$\Delta_{\partial}$-harmonic. Hence it is $\Delta_{\dbar}$-harmonic and $
\dbar(\bar{\alpha}\chi) =0$. Applying now the $\dbar$ operator in the first equation we get $\Delta_{\dbar}(\mu)=0$, which implies $\dbar^* \mu=0$. Hence $\bar{\alpha}\chi =0$ and 
$\chi=0$. From the third equation  we get $\mu=0$. Therefore there are no nonzero solution to the kernel 
equations on a K\"ahlerian monopole. \hfill $\Box$

\small


\providecommand{\bysame}{\leavevmode\hbox to3em{\hrulefill}\thinspace}
\providecommand{\MR}{\relax\ifhmode\unskip\space\fi MR }
\providecommand{\MRhref}[2]{%
  \href{http://www.ams.org/mathscinet-getitem?mr=#1}{#2}
}
\providecommand{\href}[2]{#2}

\vspace{0.7cm}
\noindent
{Luca Scala,
Department of Mathematics,
University of Chicago,
5734 S. University Avenue,
Chicago IL 60637~USA \;
{\em Email Address:} {\tt lucascala@math.uchicago.edu}
}
\end{document}